\documentclass[11pt]{article}

\usepackage{latexsym}
\usepackage{amsmath}
\usepackage{amsfonts}
\usepackage{amssymb}
\usepackage{amsthm}
\usepackage{psfrag}
\usepackage{epsfig}
\usepackage{color}

\usepackage[tight]{subfigure}
\newcommand{\R}{\mbox{\rm I\kern-.18em R}}
\newcommand{\fR}{\mbox{\footnotesize\rm I\kern-.18em R}}
\newcommand{\sR}{\mbox{\small\rm I\kern-.18em R}}
\newcommand{\N}{\mbox{\rm I\kern-.18em N}}
\newcommand{\dist}{\mathop{\rm dist}\nolimits}

\newcommand{\<}{\langle}
\renewcommand{\>}{\rangle}

\newcommand{\bcurl}{\mathop{\rm{\bf curl}}\nolimits}

\newcommand{\curlS}[1]{\mathop{\rm curl_{\rm{#1}}}\nolimits}

\newcommand{\bcurlS}[1]{\mathop{\rm{\bf curl}_{\rm{#1}}}\nolimits}

\newcommand{\G}{\Gamma}
\newcommand{\curlG}{\curlS{\Gamma}}
\newcommand{\bcurlG}{\bcurlS{\Gamma}}

\newcommand{\CE}{{\cal E}}

\newcommand{\CQ}{{\cal Q}}

\newcommand{\bn}{{\bf n}}
\newcommand{\bt}{{\bf t}}
\newcommand{\bvarphi}{{\mbox{\boldmath $\varphi$}}}

\newcommand{\be}{\begin{equation}}
\newcommand{\ee}{\end{equation}}

\newtheorem{theorem}{Theorem}[section]

\newtheorem{remark}[theorem]{Remark}
\newtheorem{lemma}[theorem]{Lemma}
\newtheorem{corollary}[theorem]{Corollary}
\newtheorem{proposition}[theorem]{Proposition}

\title{
Discontinuous Galerkin $hp$-BEM with quasi-uniform meshes
\thanks{Supported by FONDECYT-Chile through project 1110324 and
        by Ministery of Education of Spain through project MTM2010-18427}
}

\author{
Norbert Heuer
\thanks{
Facultad de Matem\'aticas, Pontificia Universidad Cat\'olica de Chile,
Avenida Vicu\~na Mackenna 4860, Macul, Santiago, Chile,
email: {\tt nheuer@mat.puc.cl}}
\and
Salim Meddahi
\thanks{Departamento de Matem\'aticas, Facultad de Ciencias,
Universidad de Oviedo, Calvo Sotelo s/n, Oviedo, Spain,
e-mail: {\tt salim@uniovi.es}}
}

\begin{document}
%\date{\today}
\date{}
\maketitle

\bigskip
\begin{abstract}
We present and analyze a discontinuous variant of the $hp$-version of the
boundary element Galerkin method with quasi-uniform meshes. The model problem
is that of the hypersingular integral operator on an (open or closed) polyhedral
surface. We prove a quasi-optimal error estimate and conclude convergence
orders which are quasi-optimal for the $h$-version with arbitrary degree and
almost quasi-optimal for the $p$-version. Numerical results underline the theory.


\end{abstract}

%%%%%%%%%%%%%%%%%%%%%%%%%%%%%%%%%%%%%%%%%%%%%%%%%%%%%%%%%%%%%%%%%%%%%%%%%%%%%%%%
\section{Introduction}

Discontinuous approximations of solutions to boundary value problems of
second order are well established and have their advantages in comparison to
conforming methods, e.g., by providing more flexibility in the design of
discrete spaces. This applies not only to finite elements but also to the
boundary element Galerkin method (BEM), that is,
finite element approximations of solutions to boundary integral operators.
Whereas there is a huge amount of literature about discontinuous finite
elements, relatively little is known on non-conforming approximations of
boundary integral operators. In this paper we analyze a discontinuous
$hp$-approximation of the hypersingular operator governing the Laplacian.
In principle, our analysis applies to other hypersingular operators as well,
e.g. from linear elasticity and acoustics (Helmholtz equation), in the sense
that the main tools from fractional order Sobolev spaces remain identical.
But some specific estimates are indeed non-trivial for these applications.
Note also that boundary element approximations of solutions to problems governed
by weakly singular operators or by equations of the second kind can use discontinuous
approximations, they are conforming. Hypersingular operators are the only
boundary integral operators from second order elliptic problems that are
challenging when one tries to relax the continuity of basis functions.

We started our study of non-conforming approximations of hypersingular operators
with the analysis of Lagrangian multipliers \cite{GaticaHH_09_BLM} (for the
implementation of essential boundary conditions) and Crouzeix-Raviart elements
\cite{HeuerS_09_CRB}. In \cite{HealeyH_10_MBE} we extended the Lagrangian
multiplier technique to a domain decomposition method, and in
\cite{ChoulyH_NDD} we presented a domain decomposition variant based on
the Nitsche coupling that avoids the additional space needed for a Lagrangian
multiplier. All these results are only about lowest order approximations. Indeed,
their analysis makes heavy use of finite-dimension arguments (like norm
equivalences and inverse properties) and Sobolev norm estimates on fixed domains
(sub-surfaces), and do not extend to high order methods.

In this paper we study the $hp$-version with quasi-uniform meshes. Key point
of our analysis is to avoid finite-dimension techniques from previous papers
and to work with general Sobolev norm estimates (not restricted to finite-dimensional
spaces) so that on elements there is no restriction for polynomial degrees.
In this way, we are able to prove quasi-optimal convergence of a
discontinuous boundary element Galerkin method (Theorem~\ref{thm_cea})
that holds on conforming quasi-uniform meshes for general non-uniform polynomial
degree distributions.
The small price to pay is that this Cea estimate involves three different
Sobolev norms ($L^2$, $H^{1/2}$ and $H^s$ with $s>1/2$). Nevertheless, it
leads to a quasi-optimal estimate in the case of the $h$-version (Corollary~\ref{cor_h})
and an almost quasi-optimal estimate for the $p$-version (Corollary~\ref{cor_p}).
In the latter case, almost quasi-optimal means that there is a $\log^{3/2}(p)$-perturbation
of the quasi-optimal estimate for the conforming method.
Though, such logarithmic perturbations are
known from the previous results (except for Crouzeix-Raviart elements) and it is
remarkable that our $h$-estimate is free from this.

Principal techniques used in the analysis are of the Strang type (discrete ellipticity
and consistency of the discrete formulation). In this sense, the approach follows
the same lines as our analysis of a Nitsche domain decomposition \cite{ChoulyH_NDD},
nevertheless relying on local instead of global estimates. For ease of presentation,
we assume that meshes are conforming. However, techniques from domain decomposition
in \cite{ChoulyH_NDD} can be used without difficulty to extend our analysis to
meshes which are conforming only on sub-domains (sub-surfaces). In this case, an additional
logarithmic perturbation appears in the error estimate.

The remainder of this paper is structured as follows. In the next section we
define some Sobolev norms, present our model problem (on a closed polyhedral
surface), specify the standard boundary element (Galerkin) method,
and recall an integration-by-parts formula for the hypersingular operator.
In Section~\ref{sec_DG} we present our discontinuous
boundary element method, state the main results (quasi-optimal convergence,
Theorem~\ref{thm_cea}) and conclude convergence orders for the lowest-order $h$-version
(Corollary~\ref{cor_h}), the $p$-version (Corollary~\ref{cor_p}), and the $h$-version
with arbitrary polynomial degree (Corollary~\ref{cor_hp}).
For presentation of the last two corollaries
we need to recall results on the regularity of the solution to our model problem
(in terms of appearing singularities) which is done in the same section.
Technical details and the proof of the main theorem are given in Section~\ref{sec_proofs}.
In Subsection~\ref{sec_open} we discuss the changes which are necessary to analyze
the model problem on an open polyhedral surface. In fact, there is some improvement
for a single smooth surface piece.
In Section~\ref{sec_num} we present several numerical results that confirm our
error estimates.

Throughout the paper, $a\lesssim b$ means that $a\le cb$ with a generic constant $c>0$
that is independent of involved parameters like $h$ or $p$. Similarly, the notation
$a\gtrsim b$ and $a\simeq b$ is used.

%%%%%%%%%%%%%%%%%%%%%%%%%%%%%%%%%%%%%%%%%%%%%%%%%%%%%%%%%%%%%%%%%%%%%%%%%%%%%%%%
\section{Sobolev spaces} \label{sec_sobolev}
\setcounter{equation}{0}
\setcounter{figure}{0}
\setcounter{table}{0}

We consider standard Sobolev spaces where the following norms are used:
For $\Omega\subset\R^n$ and $0<s<1$ we define
\[
   \|u\|^2_{H^s(\Omega)}:=\|u\|^2_{L^2(\Omega)} + |u|^2_{H^s(\Omega)}
\]
with semi-norm
\[
    |u|_{H^s(\Omega)} := 
    \Bigl(
    \int_\Omega \int_\Omega \frac{|u(x)-u(y)|^2}{|x-y|^{2s+n}} \,dx\,dy
    \Bigr)^{1/2}.
\]
For a Lipschitz domain $\Omega$ and $0<s<1$, the space
$\tilde H^s(\Omega)$ is defined as the completion of $C_0^\infty(\Omega)$
under the norm
\[
   \|u\|_{\tilde H^s(\Omega)}
   :=
   \Bigl(
   |u|^2_{H^{s}(\Omega)}
   +
   \int_\Omega \frac{u(x)^2}{(\dist(x,\partial\Omega))^{2s} } \,dx
   \Bigr)^{1/2}.
\]
For $s\in (0,1/2)$, $\|\cdot\|_{\tilde H^s(\Omega)}$ and $\|\cdot\|_{H^s(\Omega)}$
are equivalent norms whereas for $s\in(1/2,1)$ there holds
$\tilde H^s(\Omega) = H_0^s(\Omega)$, the latter space being the completion
of $C_0^\infty(\Omega)$ with norm in $H^s(\Omega)$.
For $s>0$ the spaces $H^{-s}(\Omega)$ and $\tilde H^{-s}(\Omega)$ are the
dual spaces (with $L^2(\Omega)$ as pivot space)
of $\tilde H^s(\Omega)$ and $H^s(\Omega)$, respectively.
For norms of vector valued functions we use the same notation as for scalar functions.

Let $\G$ be a piecewise plane, open or closed, polyhedral surface with faces
$\G_j$, $j=1,\ldots,L$. Throughout the paper we will identify faces with polygonal
subsets of $\R^2$. First we will analyze the case of a closed surface, and in
Section~\ref{sec_open} we will mention particular changes which are necessary
to analyze the case of an open surface.

Our model problem is:
{\em For a given function $f\in L^2(\G)$ find $u\in H^{1/2}(\G)$ such that
\(
   \<u,1\>_\G=0
\)
and}
\be \label{IE}
   Wu(x):=-\frac 1{4\pi}\frac {\partial}{\partial \bn(x)}
              \int_\G u(y) \frac {\partial}{\partial \bn(y)} \frac 1{|x-y|}
              \,dS(y)
   = f(x),\quad x\in\G.
\ee
Here, $\bn$ is the exterior normal unit vector on $\G$, and $\<\cdot,\cdot\>_\G$
denotes the $L^2(\G)$-inner product and its extension by duality.
%pairing between $H^{-1/2}(\G)$ and $\tilde H^{1/2}(\G)$. 
Throughout, this generic notation will be used for other domains/surfaces as well,
indicated by the index.

A variational formulation of (\ref{IE}) is:
{\em Find $u\in H^{1/2}(\G)$ such that
\(
   \<u,1\>_\G=0
\)
and}
\be \label{weak_org}
   \<Wu, v\>_\G
   = \<f, v\>_\G\quad\forall v\in H^{1/2}(\G).
\ee

A standard (i.e. conforming) boundary element method for the approximate solution of
(\ref{weak_org}) is to select a piecewise polynomial subspace
$H_{hp,\mathrm{conf}}\subset H^{1/2}(\G)$ and
\[
   H_{hp,\mathrm{conf}}^0 := \{v\in H_{hp,\mathrm{conf}};\; \<v,1\>_\G=0\},
\]
and to define an approximant $u_{hp,\mathrm{conf}}\in H_{hp,\mathrm{conf}}^0$ by
\be \label{bem}
\<W u_{hp,\mathrm{conf}}, v\>_\G
   = \<f, v\>_\G\quad\forall v \in H_{hp,\mathrm{conf}}^0.
\ee
The bilinear form with hypersingular operator $W$ is usually calculated by
making use of its relation to a bilinear form with weakly singular operator $V$
defined by
\[
   V\bvarphi(x) := \frac 1{4\pi}\int_\G \frac {\bvarphi(y)}{|x-y|}\,dS(y),
   \quad \bvarphi\in (H^{-1/2}(\G))^3,\ x\in\G.
\]
There holds the relation between the operators $W$ and $V$:
\be \label{WV}
   W = \curlG ( V \bcurlG )
\ee
as a linear continuous mapping from
$H^{1/2}(\G)$ to $H^{-1/2}(\G)$ (see \cite{Nedelec_82_IEN}), so that
\[
 \<W u , v\>_\G = \<\bcurlG u ,V \bcurlG v \>_\G
  \quad \forall u,v \in H^{1/2}(\G),
\]
cf. also \cite[Lemma 2.3]{GaticaHH_09_BLM}.
Here, $\bcurlG$ is the surface curl operator and $\curlG$ its adjoint operator.
In the following, $\bcurlS{Q}$ will denote the restriction of $\bcurlG$ onto
$Q\subset\G$ and we will use $\bcurlS{\hat Q}$ as the surface curl (rotated gradient)
on a reference element $\hat Q$ in local coordinates.

Let us recall from \cite{GaticaHH_09_BLM} that integration by parts
for $\bcurlG$ on any face $\G_j$ gives rise to a linear bounded operator
\be \label{T_global}
      \left\{\begin{array}{cll}
      \{v\in H^{1/2}(\G);\; Wv\in\tilde H^{-1/2}(\G_j)\} &\to
      & H^{-1/2}(\partial\G_j)\\
      v &\mapsto & \bt_j\cdot (V\bcurlG v)|_{\partial\G_j}
      \end{array}\right.
\ee
with
\[
   \<Wv, w\>_{\G_j}
   =
   \<\bcurlS{\G_j} w, V\bcurlG v\>_{\G_j}
   +
   \<\bt_j\cdot V\bcurlG v, w\>_{\partial\G_j}
   \quad\forall w\in H^1(\G_j).
\]
Here, $\bt_j$ is the unit tangential vector along the boundary of $\G_j$,
with mathematically positive orientation which is compatible with the direction
of the normal vector $\bn$.
In fact, it is easy to see that the operator \eqref{T_global} maps to
$H^{-\epsilon}(\partial\G_j)$ for any $\epsilon>0$ so that the test functions
$w$ in \eqref{IP} can be less regular (we don't make use of this here).
Below, we need an integration-by-parts formula on sub-domains (sub-surfaces)
$Q\subset\G$. In order to have a well-defined bilinear form $\<Wv, w\>_Q$
for any sufficiently regular subset $Q$ of $\G$ we assume that $Wv\in L^2(\G)$.
Then analogous arguments prove that for any Lipschitz surface $Q\subset\G$
with unit tangential vector $\bt_Q$ (with mathematically positive orientation)
there holds
\be \label{T_local}
      \left\{\begin{array}{cll}
      \{v\in H^{1/2}(\G);\; Wv\in L^2(\G)\} &\to
      & H^{-1/2}(\partial Q)\\
      v &\mapsto & \bt_Q\cdot (V\bcurlG v)|_{\partial Q}
      \end{array}\right.
\ee
with
\be \label{IP}
   \<Wv, w\>_Q
   =
   \<\bcurlS{Q} w, V\bcurlG v\>_Q
   +
   \<\bt_Q\cdot V\bcurlG v, w\>_{\partial Q}
   \quad\forall w\in H^1(Q).
\ee

%%%%%%%%%%%%%%%%%%%%%%%%%%%%%%%%%%%%%%%%%%%%%%%%%%%%%%%%%%%%%%%%%%%%%%%%%%%%%%%%
\section{Discontinuous Galerkin method}
\label{sec_DG}
\setcounter{equation}{0}
\setcounter{figure}{0}
\setcounter{table}{0}

Let $\CQ_h$ be a conforming quasi-uniform mesh which is compatible with the
faces of $\G$ and whose (closed) elements are shape-regular triangles or
quadrilaterals $Q$.
The maximum diameter of the elements is denoted by $h$.
The set of edges of $\CQ_h$ is denoted by $\CE_h$.
We also need the skeleton of the mesh
\[
   \gamma_h := \cup_{e\in\CE_h}\bar e.
\]
For sufficiently smooth functions $v$ on $\G$ we define jumps
%and averages
as follows.
For $e\in\CE_h$ with $e=Q_1\cap Q_2$, $Q_1,Q_2\in\CQ_h$,
\be \label{jump}
   [v]|_e := (v|_{Q_1}-v|_{Q_2})|_e.
   %\quad \{v\}|_e := \frac 12(v|_{Q_1}+v|_{Q_2})|_e
\ee
Here, the selection of the numbering
$Q_1$, $Q_2$ is arbitrary but fixed.
%Jumps and averages of vector functions are defined componentwise.

We also assign unit tangential vectors,
$\bt_Q$ on the boundary of $Q\in\CQ_h$.
On edges, unit tangential vectors are inherited
(using the association $e\mapsto Q_1$ previously mentioned):
\be \label{t_e}
   \bt_e:= \bt_{Q_1}|_e\quad  e\in\CE_h.
\ee
We define the broken surface curl operator for sufficiently smooth functions by
\[
   (\bcurl_h v)|_Q := \bcurlS{Q}(v|_Q)\quad\forall Q\in\CQ_h.
\]
We use discontinuous $hp$-spaces with variable degree for discretization.
For a triangle $Q\in\CQ_h$ let $P_p(Q)$ denote the space of polynomials
of total degree $p$, and for quadrilaterals let $P_p(Q)$ be the space
of (bi)linearly transformed polynomials of degree $p$ in both coordinates
on a reference square. Throughout, we will use the generic notation $\hat Q$ for
a reference element (the unit square for quadrilaterals and a fixed triangle for
triangular elements).

For given polynomial degrees $p=p(Q)$ (independent for different elements $Q$)
we introduce discrete spaces
\[
	X_{hp} := \{v\in L^2(\G);\; v|_Q\in P_{p(Q)}(Q)\ \forall Q\in\CQ_h\}
\]
and the subspaces
\[
	X_{hp}^0 := \{v\in X_{hp};\; \<v,1\>_\G=0\}.
\]
The discontinuous Galerkin boundary element method then reads:

{\it Find $u_{hp} \in X_{hp}^0$ such that}
\be \label{DGBEM}
   a_h(u_{hp},v) = \<f, v\>_\G \quad\forall v\in X_{hp}^0.
\ee
Here,
\be \label{bil}
  a_h(v,w)
  :=
  \<V \bcurl_h v, \bcurl_h w\>_\G
  + \<Tv, [w]\>_{\gamma_h} - \<[v], Tw\>_{\gamma_h}
  + \nu \<[v], [w]\>_{\gamma_h}
\ee
where $\nu > 0$ is a given parameter.
%and $\sigma\in\{-1,1\}$ are given parameters.
Furthermore, the operator $T$ is defined as
\be \label{T}
   (T v)|_{e} := (V \bcurl_h v)|_e \cdot \bt_e \qquad (e\in\CE_h).
\ee
It goes without further discussion that the integral mean zero condition
in $X_{hp}^0$ can be implemented by a rank one Lagrangian multiplier.

Note that $\bcurl_h v\in L^2(\G)$ for $v\in X_{hp}$ so that the tangential
trace of $V\bcurl_h v\in H^1(\G)$ is well defined
without employing formula \eqref{IP}.

The analysis will be based on ``broken'' Sobolev (semi-) norms
\[
   |v|_{H^s(\CQ_h)}^2 := \sum_{Q\in\CQ_h} |v|_{H^s(Q)}^2,
   \qquad
   \|v\|_{H^s_\nu(\CQ_h)}^2 :=
   |v|_{H^s(\CQ_h)}^2  + \nu \|[v]\|_{L^2(\gamma_h)}^2 + |\int_\G v|^2
   \quad (1/2 \le s\le 1).
\]

\begin{theorem} \label{thm_cea}
Let $u\in H^r(\G)$ with $r\in (1/2,3/2)$ be the solution of \eqref{weak_org},
and let $\nu>0$, $\delta\in (0,1/2)$. Then, the discrete problem \eqref{DGBEM}
is uniquely solvable and there exists a constant $C>0$, depending on
$r$ and $\delta$, but not on $\nu$, $u$, the actual mesh and polynomial degrees,
such that for any $s\in(1/2,\min\{r,1-\delta\}]$
there holds the quasi-optimal error estimate
\begin{multline*}
   \|u-u_{hp}\|_{H^{1/2}_\nu(\CQ_h)}
   \le C
   \\
   \inf_{v\in X_{hp}^0\cap C^0(\G)}
   \Bigl\{
      \|u-v\|_{H^{1/2}(\G)}
      +
      \frac 1{\nu^{1/2}(s-1/2)^{3/2}}
      \Bigl(
         h^{-1/2} \|u-v\|_{L^2(\G)}
         +
         h^{s-1/2} \|u-v\|_{H^s(\G)}
      \Bigr)
   \Bigr\}.
\end{multline*}
\end{theorem}

A proof of this result will be given at the end of Section~\ref{sec_proofs}.
The parameter $\delta$ above is needed to maintain a fixed positive distance
of $s$ to $1$ so that norms involved in the proof are uniformly equivalent and the
constant $C$ does not depend on $s$.
This will be needed for the estimate of the $p$-version when $s\to 1/2$,
see Corollary~\ref{cor_p} below.
First let us consider the $h$-version with quasi-uniform meshes and piecewise
(bi)linear functions. Standard approximation results prove that in this case
the discontinuous Galerkin boundary element method performs as well as
the standard (conforming) BEM.

\begin{corollary}[$h$-version lowest degree] \label{cor_h}
Let $p=1$. For $u\in H^r(\G)$ with $r<3/2$ there holds
\[
   \|u-u_{hp}\|_{H^{1/2}_\nu(\CQ_h)}
   \lesssim
   (1+\nu^{-1/2}) h^{r-1/2}\|u\|_{H^r(\G)}.
\]
\end{corollary}

Fixing a mesh and improving the approximation by increasing polynomial degrees
leads to the $p$-version. Convergence analysis for conforming methods and
based on standard Sobolev regularity proves convergence orders which are analogous
to those of the $h$-version (regularity order minus $1/2$), but without upper bound.
In the case of polyhedral surfaces,
regularity is limited and depends on the angles at edges and corners where singularities
appear. Specific approximation analysis for these singularities show that convergence
orders are twice that of the $h$-version when element boundaries are aligned with edges,
as is the case in this paper. In order to claim and prove a precise error estimate
we would need to recall the corresponding regularity results by
Dauge, von Petersdorff and Stephan \cite{Dauge_88_EBV,vP,vonPetersdorffS_90_RMB}.
In order to improve readability, and following the results from these references,
we reduce this presentation to the assumption
that the solution $u$ of \eqref{weak_org} can be written like
\be \label{reg}
   u = u_{\rm reg} + \sum_{\gamma\in\{\mbox{\footnotesize edges}\}} u^\gamma
               + \sum_{v\in\{\mbox{\footnotesize vertices}\}} u^v
               + \sum_{\gamma\in\{\mbox{\footnotesize edges}\},\;
                       v\in\{\mbox{\footnotesize vertices}\},\;
                       v\in\bar\gamma} u^{\gamma v}
\ee
where $u^\gamma$, $u^v$ and $u^{\gamma v}$ are the edge-, vertex-, and edge-vertex
singularities of $u$, and $u^{reg}$ is a remainder of higher regularity.
The sets $\{\mbox{edges}\}$ and $\{\mbox{vertices}\}$ denote the edges and vertices
of the polyhedral surface $\G$. The singularities are of the types (ignoring
cut-off functions and smoother parts)
\begin{align*}
   u^\gamma &= |\log \dist(\cdot,\gamma)|^l \dist(\cdot,\gamma)^\mu
               &&\hspace{-4em}(l\in\{0,\ldots,s\},\; s\in\N\cup\{0\},\; \mu > 1/2)
   \\
   u^v &= |\log \dist(\cdot,v)|^t \dist(\cdot,v)^\lambda
          &&\hspace{-4em}(t\in\{0,\ldots,q\},\; q\in\N\cup\{0\},\; \lambda > 0)
   \\
   u^{\gamma v} &= |\log \dist(\cdot,\gamma_1)|^l\; |\log \dist(\cdot,\gamma_2)|^t
                     \dist(\cdot,\gamma_1)^{\lambda-\mu}
                     \dist(\cdot,\gamma_2)^\mu
                     &&(l\in\{0,\ldots,s\},\; t\in\{0,\ldots,q\},
                 \\ & &&\ \mu>1/2,\;\lambda>0)
\end{align*}
where $\gamma_1$ and $\gamma_2$ denote the edges that meet at the vertex $v$.
The smallest parameters $\mu$, $\lambda$, and largest number $s$, $q$ limit
the convergence order of the $p$-version.

The following result shows that the $p$-version of the discontinuous Galerkin BEM
converges almost as fast as the conforming version: the error estimate behaves
the same asymptotically except for an additional $\log^{3/2} p$ factor which, however,
can be compensated by choosing $\nu\sim (\log p)^3$.

\begin{corollary}[$p$-version] \label{cor_p}
Let $X_{hp}$ be defined by a fixed mesh that consists of triangles and/or
parallelograms and let $p$ denote the minimum polynomial degree. We assume
that the solution $u$ of \eqref{weak_org} is of the type \eqref{reg} with
minimum parameters $\mu$ and $\lambda$ and maximum parameters $q$, $s$
and that the regular part $u_{\rm reg}$ is as smooth as needed, which can be
achieved by considering sufficiently many singularity terms.
Then there holds
\[
   \|u-u_{hp}\|_{H^{1/2}_\nu(\CQ_h)}
   \lesssim
   (1 + \nu^{-1/2} \log^{3/2}(1+p)) p^{-2\min\{\lambda+1/2,\mu\}} \log^\beta (1+p),
\]
where
\[
   \beta=\left\{\begin{array}{ll}
      s+q + 1/2 & \mbox{if}\ \lambda=\mu-1/2, \\
      s+q       & \mbox{otherwise}.
   \end{array}\right.
\]
\end{corollary}

\begin{proof}
By Theorem~\ref{thm_cea} there holds for a fixed mesh
\[
   \|u-u_{hp}\|_{H^{1/2}_\nu(\CQ_h)}
   \lesssim
   \inf_{v\in X_{hp}^0\cap C^0(\G)}
   \Bigl(\|u-v\|_{H^{1/2}(\G)} + \nu^{-1/2}(s-1/2)^{-1} \|u-v\|_{H^s(\G)}\Bigr)
\]
and by \cite[Theorem 9.1]{SchwabS_96_OpA} we can find $v\in X_{hp}^0\cap C^0(\G)$
such that
\[
   \|u-v\|_{H^s(\G)} \lesssim p^{-2(\alpha - s)} \log^\beta (1+p)
\]
with $\alpha=\min\{\lambda+1,\mu+1/2\}$ (the integral mean zero condition can
be added in a postprocessing step and does not diminish the convergence rate).
This bound holds for $s\in [0,1]$ by interpolation.
Choosing $s=1/2$ and $s=1/2+\log^{-1}(1+p)$ (where the interpolation norm is uniformly
equivalent to the Sobolev-Slobodeckij norm; for details see
Section~\ref{sec_proofs} and, in particular, Remark~\ref{rem_norms})
and combining both bounds proves
\[
   \|u-u_{hp}\|_{H^{1/2}_\nu(\CQ_h)}
   \lesssim
   (1 + \nu^{-1/2} \log^{3/2}(1+p)) p^{-2(\alpha-1/2)} \log^\beta (1+p),
\]
which is the assertion.
\end{proof}

\begin{remark} \label{rem_p}
In theory, the additional factor of $\log^{3/2} p$ in the $p$-estimate by
Corollary~\ref{cor_p} can be compensated by choosing
$\nu\sim \log^3 p$. The $\log^{3/2} p$ perturbation in the error
estimate is due to the factor $(s-1/2)^{-3/2}$ in the abstract error estimate
of Theorem~\ref{thm_cea}. In the case of an open smooth surface (one face),
there appears a factor $(s-1/2)^{-1/2}$ rather than $(s-1/2)^{-3/2}$,
see Theorem~\ref{thm_cea_plane}. Therefore,
in this case, the $p$-estimate has only a perturbation of $(\log p)^{1/2}$
which can be compensated by choosing $\nu\sim\log p$ instead of $\nu\sim \log^3 p$
as in the case of multiple faces.
Of course, these are only upper bounds which are not known to be exact and
this phenomenon is difficult to observe numerically.
\end{remark}

By using $hp$-approximation results (for quasi-uniform meshes) from
\cite{BespalovH_08_hpB} one can easily deduce error estimates
for the corresponding $hp$-version of the discontinuous Galerkin BEM.
Nevertheless, in order to keep the presentation
simple, we restrict ourselves to the case of arbitrary but fixed polynomial degrees.

\begin{corollary}[$h$-version with arbitrary polynomial degree] \label{cor_hp}
Let the subspaces $X_{hp}$ be defined by a sequence of meshes with fixed (lowest) polynomial
degree $p$. We assume
that the solution $u$ of \eqref{weak_org} is of the type \eqref{reg} with
minimum parameters $\mu$ and $\lambda$ and maximum parameters $q$, $s$
and that the regular part $u_{\rm reg}\in H^1(\G)\cap\Pi_{j=1}^L H^k(\G_j)$.
Then there holds
\[
   \|u-u_{hp}\|_{H^{1/2}_\nu(\CQ_h)}
   \lesssim
   (1+\nu^{-1/2})
     \max\left\{
       h^{\min\,\{k-1/2,p+1/2\}},
       h^{\min\,\{\lambda+1/2,\mu\}}\,
       (1+|\log h|)^{\beta_1+\beta_2}
     \right\}
\]
where
\[
   \beta_1=\left\{\begin{array}{ll}
      s+q + 1/2 & \mbox{if}\ \lambda=\mu-1/2, \\
      s+q       & \mbox{otherwise}
   \end{array}\right.
\]
and
\[
   \beta_2 = \left\{\begin{array}{ll}
          \frac 12 &
          \hbox{if \ $p = \min\,\{\lambda,\,\gamma-\frac 12\}$},\\
   \noalign{\vskip3pt}
          0          & \hbox{otherwise.}\\
   \end{array}\right.
\]
\end{corollary}

\begin{proof}
By \cite[Theorem 7.1]{BespalovH_08_hpB} there exists $v\in X_{hp}\cap C^0(\G)$
such that
\[
     \|u-v\|_{H^s(\G)}
     \lesssim
     \max\left\{
       h^{\min\,\{k,p+1\}-s},
       h^{\min\,\{\lambda+1,\mu+1/2\}-s}\,
       (1+|\log h|)^{\beta_1+\beta_2}
     \right\}.
\]
for $0\le s<\min\{1,\lambda+1,\mu+1/2\}$.
Using this result together with the bound from Theorem~\ref{thm_cea} proves
the assertion.
\end{proof}

%%%%%%%%%%%%%%%%%%%%%%%%%%%%%%%%%%%%%%%%%%%%%%%%%%%%%%%%%%%%%%%%%%%%%%%%%%%%%%%%
\section{Technical results and proof of the main theorem}
\label{sec_proofs}
\setcounter{equation}{0}
\setcounter{figure}{0}
\setcounter{table}{0}

Throughout we will make use of the continuity and ellipticity of $V$
proved by Costabel for Lipschitz surfaces \cite{Costabel_88_BIO}:
\be \label{V_cont}
   V:\; H^{s-1}(\G) \to H^{s}(\G), \quad 0 \le s \le 1,
\ee
\be \label{V_posdef}
   \<V v, v\>_\G
   \gtrsim
   \|v\|_{H^{-1/2}(\G)}^2\qquad\forall v\in H^{-1/2}(\G).
\ee
We will need several technical results involving fractional order Sobolev norms.
So far we have defined the Sobolev-Slobodeckij norm. However, standard tool
to prove estimates in Sobolev spaces is interpolation theory. We therefore have
to deal with at least two different definitions which, usually, give rise to
equivalent norms, cf.~Remark~\ref{rem_norms} below. We use the K-method of
interpolation (see, e.g., \cite{LionsMagenes} for details).
For $\Omega\subset\R^n$ and $0<s<1$, the following spaces
can be equivalently defined (i.e., norms are equivalent)
\[
   H^s(\Omega) = \Big(L^2(\Omega), H^1(\Omega)\Big)_{s,2},\qquad
   \tilde H^s(\Omega) = \Big(L^2(\Omega), H_0^1(\Omega)\Big)_{s,2}.
\]
For negative $s$, we also need spaces defined by interpolation, i.e.
\[
   %H^s(\Omega) = \Big(H^{-1}(\Omega), L^2(\Omega)\Big)_{1+s,2},\qquad
   \tilde H^s(\Omega) = \Big(\tilde H^{-1}(\Omega), L^2(\Omega)\Big)_{1+s,2}
   \qquad (-1<s<0).
\]

\begin{remark} \label{rem_norms}
For fixed domain and fixed order, it is well known that norms in Sobolev spaces
defined by interpolation and by the double integral (Sobolev-Slobodeckij) are
equivalent. In order to avoid dependence on a variable domain (elements)
we use the equivalence on a fixed domain (reference element) by previous
transformation. We will need this equivalence for varying order $s$ close to $1/2$
where it is uniform, cf.~\cite[Proof of Corollary 4]{Heuer_01_ApS}.
%Furthermore, we need the equivalence of negative order norms ($s$ close to $-1/2$)
%defined by duality and by interpolation. This is again uniform when $s$ stays
%away from the integers $-1$, $0$, cf.~\cite[Proof of Corollary 2]{Heuer_01_ApS}.
\end{remark}

In order to be transparent, we indicate the type of norm used in the technical lemmas
below (i.e., specifying ``Sobolev-Slobodeckij'' or ``K-method'' or ``duality'').

\begin{lemma} \label{la_norms}
Let $\hat Q$ be the generic reference element
and $R$ a fixed Lipschitz domain (possibly $\hat Q$) with boundary $\partial R$.\\
(i) (Sobolev-Slobodeckij)
For any $s \in (-1/2,1/2)$, and any $v \in H^s(R)$, there holds
\be \label{equiv}
   \|v\|_{\tilde H^s(R)} \lesssim \frac{1}{1/2 - |s|} \|v\|_{H^s(R)}.
\ee
(ii) (Sobolev-Slobodeckij) For any $s \in (1/2,1]$, and any $v \in H^s(R)$, there holds
\be \label{trace}
   \|v\|_{L^2(\partial R)}
   \lesssim
   \frac{1}{\sqrt{s - 1/2}} \|v\|_{H^s(R)}.
\ee
(iii)
For a (bi)linear mapping $M_Q:\hat Q\to Q$ onto an element
$Q\in\CQ_h$ there holds
\be \label{scale1}
   \|v\|_{\tilde H^{s}(Q)}^2
   \simeq
   h^{2-2s} \|v\circ M_Q\|_{\tilde H^{s}(\hat Q)}^2
   \quad (v\in \tilde H^{s}(Q),\ s\in [0,1],\ \mbox{K-method}),
\ee
\be \label{scale2}
   |v|_{H^{s}(Q)}^2
   \simeq
   h^{2-2s} |v\circ M_Q|_{H^{s}(\hat Q)}^2
   \quad (v\in H^{s}(Q),\ s\in [0,1],\ \mbox{Sobolev-Slobodeckij})
\ee
and
\be \label{scale3}
   \|v\|_{H^{s}(Q)}^2
   \simeq
   h^{2-2s} \|v\circ M_Q\|_{H^{s}(\hat Q)}^2
   \quad (v\in H^{s}(Q),\ s\in [-1,0],\ \mbox{duality}).
\ee
\end{lemma}

\begin{proof}
Assertion (i) is shown by \cite[Lemma~5]{Heuer_01_ApS} and
(ii) is \cite[Lemma 4.3]{GaticaHH_09_BLM}. Equivalence \eqref{scale2} follows
from the definition of the semi-norm.
For \eqref{scale1} and \eqref{scale3} in the case of a reference square and
affine maps see \cite[Lemma~2]{Heuer_01_ApS}.
The cases of triangular and quadrilateral non-parallelogram elements follow similarly.
\end{proof}

\begin{lemma} \label{la_DD}
(K-method)
For any $s\in[-1,1]$ there holds
\be \label{DD1}
   \sum_{Q\in\CQ_h} \|v\|_{H^s(Q)}^2 \lesssim \|v\|_{H^s(\G)}^2
   \qquad \forall v\in H^s(\G),
\ee
\be \label{DD2}
   \|v\|_{H^s(\G)}^2 \lesssim \sum_{j} \|v\|_{\tilde H^s(\G_j)}^2
   \quad\forall v\in H^s(\G):\ v|_{\G_j}\in \tilde H^s(\G_j)\;\forall j.
\ee
\end{lemma}

\begin{proof}
For complex interpolation these estimates have been analyzed by
von Petersdorff \cite{vP}; for the K-method see \cite{AinsworthMT_99_CBE},
cf. also \cite{Heuer_01_ApS}.
\end{proof}

\begin{lemma} \label{la_curl}
(i)
\be \label{curl_closed}
   \|\bcurlS{\hat Q} v\|_{H^{-1/2}(\hat Q)}
   \gtrsim
   |v|_{H^{1/2}(\hat Q)}
   \quad\forall v\in H^{1/2}(\hat Q)
\ee
(ii)
\be \label{curl_cont_G}
   \bcurlG:\; H^{1/2}(\G) \to \Bigl(H^{-1/2}(\G)\Bigr)^3
\ee
(iii) (K-method) For $s\in [0,1]$ there holds
\be \label{curl_cont_face_tilde}
   \bcurlS{\G_j}:\; \tilde H^s(\G_j) \to \Bigl(\tilde H^{s-1}(\G_j)\Bigr)^2,
   \quad j=1,\ldots,L,
\ee
and
\be \label{curl_cont_face_nontilde}
   \bcurlS{\G_j}:\; H^s(\G_j) \to \Bigl(H^{s-1}(\G_j)\Bigr)^2,
   \quad j=1,\ldots,L.
\ee
%The continuity is uniform in $s$ if $s\in [\delta,1-\delta]$ for fixed $\delta\in (0,1)$.
\end{lemma}

\begin{proof}
(i) is \cite[Lemma 4.1]{GaticaHH_09_BLM} and for (ii) we refer to \cite{BuffaCS_02_THL}.
In \eqref{curl_cont_G} we simply put a norm for vector functions with three components.
More precisely, $\bcurlG$ maps onto a vector space of tangential fields.
For details we refer to \cite{BuffaCS_02_THL}.
On individual faces, where $\G_j$ is identified with a subset of $\R^2$,
this tangential space is simply a vector space of functions with two components.
In order to show \eqref{curl_cont_face_tilde} we recall that $\G_j$ is being identified
with a subset of $\R^2$ and, thus, we can extend $v\in L^2(\G_j)$ by $0$ to
$v^0\in L^2(\R^2)$. By Fourier analysis,
$\bcurlS{\fR^2}:\; L^2(\R^2)\to \bigl(H^{-1}(\R^2)\bigr)^2$,
so that by the density of smooth functions with compact support in $L^2(\G_j)$
\be \label{pf_la_curl_1}
   \|\bcurlS{\G_j} v\|_{\tilde H^{-1}(\G_j)}
   \lesssim
   \|\bcurlS{\fR^2} v^0\|_{H^{-1}(\fR^2)}
   \lesssim
   \|v^0\|_{L^2(\fR)}
   =
   \|v\|_{L^2(\G_j)}.
\ee
Here, the estimate $\|\psi\|_{\tilde H^{-1}(\G_j)}\lesssim \|\psi^0\|_{H^{-1}(\fR^2)}$
for any $\psi\in C_0^\infty(\G_j)$ with zero extension $\psi^0$ follows from
the existence of a bounded extension operator
\(
   E_1:\; H^1(\G_j)\to H^1(\R^2):
\)
\begin{align*}
   \|\psi\|_{\tilde H^{-1}(\G_j)}
   &=
   \sup_{\varphi\in H^1(\G_j)}
   \frac {\<\psi,\varphi\>_{\G_j}}{\|\varphi\|_{H^1(\G_j)}}
   =
   \sup_{\varphi\in H^1(\G_j)}
   \frac {\<\psi^0,E_1\varphi\>_{\fR^2}}
         {\|\varphi\|_{H^1(\G_j)}}
   \\
   &\lesssim
   \sup_{\varphi\in H^1(\G_j)}
   \frac {\<\psi^0,E_1\varphi\>_{\fR^2}}
         {\|E_1\varphi\|_{H^1(\fR^2)}}
   \le
   \sup_{\varphi\in H^1(\fR^2)}
   \frac {\<\psi^0,\varphi\>_{\fR^2}}
         {\|\varphi\|_{H^1(\fR^2)}}
   =
   \|\psi^0\|_{H^{-1}(\fR^2)}.
\end{align*}
Since also $\bcurlS{\G_j}:\; H^1_0(\G_j)\to \bigl(L^2(\G_j)\bigr)^2$, the assertion
\eqref{curl_cont_face_tilde}
%for $\tilde H^s(\G_j)$ and $\tilde H^{s-1}(\G_j)$ with interpolation norms
follows by interpolation.
The boundedness \eqref{curl_cont_face_nontilde} is proved by interpolation
between
\(
   \bcurlS{\G_j}:\; H^1(\G_j)\to \bigl(L^2(\G_j)\bigr)^2
\)
and
\(
   \bcurlS{\G_j}:\; L^2(\G_j)\to \bigl(H^{-1}(\G_j)\bigr)^2.
\)
The latter boundedness follows from \eqref{pf_la_curl_1} by noting that
\[
   \|\bcurlS{\G_j} v\|_{H^{-1}(\G_j)}
   \lesssim
   \|\bcurlS{\G_j} v\|_{\tilde H^{-1}(\G_j)}.
%   \|\bcurlS{\fR^2} v^0\|_{H^{-1}(\fR^2)}
\]
%Specifically, using extension by zero
%\(
%   (\cdot)^0:\; L^2(\G_j)\to L^2(\R^2)
%\)
%there holds for any $\psi\in C_0^\infty(\G_j)$
%\begin{align*}
%   \|\psi\|_{H^{-1}(\G_j)}
%   &=
%   \sup_{\varphi\in H^1_0(\G_j)}
%   \frac {\<\psi,\varphi\>_{\G_j}}{\|\varphi\|_{H^1_0(\G_j)}}
%   =
%   \sup_{\varphi\in H^1_0(\G_j)}
%   \frac {\<\psi^0,\varphi^0\>_{\fR^2}}
%         {\|\varphi\|_{H^1_0(\G_j)}}
%   \\
%   &\lesssim
%   \sup_{\varphi\in H^1_0(\G_j)}
%   \frac {\<\psi^0,\varphi^0\>_{\fR^2}}
%         {\|\varphi^0\|_{H^1(\fR^2)}}
%   \le
%   \sup_{\varphi\in H^1(\fR^2)}
%   \frac {\<\psi^0,\varphi\>_{\fR^2}}
%         {\|\varphi\|_{H^1(\fR^2)}}
%   =
%   \|\psi^0\|_{H^{-1}(\fR^2)}.
%\end{align*}
%\cred{Check carefully proof of \eqref{curl_cont_face_tilde} and
%\eqref{curl_cont_face_nontilde}. Actually, \eqref{curl_cont_face_tilde} is only
%needed for $s=-1$ and \eqref{curl_cont_face_nontilde} for $s > 1/2$.}
%By the uniformity of the equivalence of
%interpolation norm for $\tilde H^{s-1}(\G_j)$ and the norm defined by duality
%(cf.~Remark~\ref{rem_norms}) .......
\end{proof}

\begin{proposition}[consistency] \label{prop_con}
Let $f\in L^2(\G)$ be given.  The DG BEM formulation is consistent, i.e.
the solution $u$ of \eqref{weak_org} solves
\[
   \<u,1\>_\G=0\qquad\mbox{and}\qquad
   a_h(u,v) = \<f, v\>_\G \qquad\forall v\in X_{hp}^0.
\]
\end{proposition}

\begin{proof}
By definition $u$ has integral mean zero.
It is well known that, for the given geometry and $f\in L^2(\G)$, $u\in H^1(\G)$.
Therefore $u$ is continuous in the sense of traces and $\bcurl_h u=\bcurlG u$ on $\G$,
and $V\bcurlG u$ is continuous in the sense of traces, as well.
Hence, recalling the signs in the definitions of $T$ \eqref{T}, the jump \eqref{jump}
and $\bt_e$ \eqref{t_e},
\begin{align*}
  a_h(u,v)
  &=
  \<V \bcurlG u, \bcurl_h v\>_\G + \<Tu, [v]\>_{\gamma_h}\\
  &=
  \sum_{Q\in\CQ_h} \<V\bcurlG u, \bcurlS{Q} v\>_Q
  +
  \sum_{e\in\CE_h} \<Tu, [v]\>_e\\
  &=
  \sum_{Q\in\CQ_h} \<V\bcurlG u, \bcurlS{Q} v\>_Q
  +
  \sum_{Q\in\CQ_h} \<\bt_Q\cdot V \bcurlG u, v\>_{\partial Q}.
\end{align*}
Integration by parts on each element proves the assertion, cf.~\eqref{IP}.
\end{proof}

\begin{proposition}[discrete ellipticity] \label{prop_ell}
There exists a constant $C>0$ independent of $\nu$ and $h$ such that
\[
   a_h(v,v)
   \gtrsim
   \Bigl( \|\bcurl_h v\|_{H^{-1/2}(\G)}
           +
           \sqrt{\nu} \|[v]\|_{L^2(\gamma_h)}
   \Bigr)
   \|v\|_{H^{1/2}_\nu(\CQ_h)}
   \quad\forall v\in X_{hp}^0.
\]
\end{proposition}

\begin{proof}
Let $v\in X_{hp}^0$ be given. Combining \eqref{V_posdef} and \eqref{DD1} we find that
there holds
\begin{align}
   \<V \bcurl_h v , \bcurl_h v\>_\G
   &\gtrsim
   \|\bcurl_h v\|_{H^{-1/2}(\G)}^2
   \label{pf_ell_1}
   \\
   &\gtrsim
   \sum_{Q\in\CQ_h} \|\bcurlS{Q} v\|_{H^{-1/2}(Q)}^2.
   \label{pf_ell_2}
\end{align}
Transforming $Q$ forth and back onto the reference element $\hat Q$,
denoting the transformed function $v$ by $\hat v$, and
employing \eqref{scale3}, \eqref{curl_closed}, \eqref{scale2}, we obtain
\[
   \|\bcurlS{Q} v\|_{H^{-1/2}(Q)}^2
   \simeq
   h^3 \|h^{-1}\bcurlS{\hat Q} \hat v\|_{H^{-1/2}(\hat Q)}^2
   \gtrsim
   h |\hat v|_{H^{1/2}(\hat Q)}^2
   \simeq
   |v|_{H^{1/2}(Q)}^2.
\]
Together with \eqref{pf_ell_2} this estimate proves that
\be \label{pf_ell_3}
   \<V \bcurl_h v , \bcurl_h v\>_\G
   \gtrsim
   \sum_{Q\in\CQ_h} |v|_{H^{1/2}(Q)}^2
   =
   |v|_{H^{1/2}(\CQ_h)}^2.
\ee
Finally, noting that
\[
   a_h(v,v) = 
   \<V \bcurl_h v, \bcurl_h v\>_\G + \nu \<[v], [v]\>_{\gamma_h},
\]
a combination of \eqref{pf_ell_1} and \eqref{pf_ell_3} together with
the definition of the norm $\|\cdot\|_{H^{1/2}_\nu(\CQ_h)}$ finishes
the proof.
\end{proof}

%%%%%%%%%%%%%%%%%%%%%%%%%%%%%%%%%%%%%%%%%%%%%%%%%%%%%%%%%%%%%%%%%%%%%%%%%%%%%%%%
\subsection{Proof of the main theorem} \label{sec_proof}
By the discrete ellipticity of the bilinear form $a_h(\cdot,\cdot)$,
there exists a unique solution $u_{hp}$ to \eqref{DGBEM} and it remains
to prove the error estimate.
%To this end, note that the order $s$ of the involved Sobolev norms
%(defined by interpolation and by the Sobolov-Slobodeckij double integrals)
%is confined to $s>1/2$ and bounded away from $1$ so that
%involved norms are uniformly equivalent, cf. Remark~\ref{rem_norms}.

We start with the standard Strang-technique of introducing a discrete function $v$,
in this case $v\in X_{hp}^0\cap C^0(\G)$, and using the
triangle inequality:
\begin{align} \label{pf_main_-2}
   \|u-u_{hp}\|_{H^{1/2}_\nu(\CQ_h)}
   &\le
   \|u-v\|_{H^{1/2}_\nu(\CQ_h)}
   +
   \|u_{hp}-v\|_{H^{1/2}_\nu(\CQ_h)}.
\end{align}
Since $[u]=[v]=0$ on $\gamma_h$ and $\<u,1\>_\G=\<v,1\>_\G=0$ there holds
\be \label{pf_main_-1}
   \|u-v\|_{H^{1/2}_\nu(\CQ_h)}
   =
   |u-v|_{H^{1/2}(\CQ_h)} \lesssim |u-v|_{H^{1/2}(\G)}
   \le
   \|u-v\|_{H^{1/2}(\G)}.
\ee
Here we made use of the bound
\be \label{DD3}
   |v|_{H^s(\CQ_h)} \le C\,|v|_{H^s(\G)}
   \quad \forall v\in H^s(\G) \qquad (s\in(0,1])
\ee
with $s=1/2$. For the Sobolev-Slobodeckij semi-norm (employed in this instance
to formulate the theorem), this follows with $C=1$ from its definition.
For the norm defined by interpolation one uses \eqref{DD1}
and an argument from quotient spaces to reproduce the semi-norm on the right-hand side.

To bound the second term on the right-hand side of \eqref{pf_main_-2} we
use the discrete ellipticity of the bilinear form $a_h(\cdot,\cdot)$,
cf.~Proposition~\ref{prop_ell}.
We obtain for any $v\in X_{hp}^0\cap C^0(\G)$
\begin{align} \label{pf_main_0}
   \|u_{hp}-v\|_{H^{1/2}_\nu(\CQ_h)}
   &\lesssim
   \sup_{w\in X_{hp}^0}
   \frac {a_h(u_{hp}-v,w)}
         {\|\bcurl_h w\|_{H^{-1/2}(\G)}
          + \sqrt{\nu}\|[w]\|_{L^2(\gamma_h)}}.
\end{align}
By Proposition~\ref{prop_con}, the definition of $u_{hp}$ and the fact that
$[u]=[v]=0$ on $\gamma_h$, there holds for any $w\in X_{hp}^0$
\begin{align} \label{pf_main_1}
   a_h(u_{hp}-v,w)
   =
   a_h(u-v,w)
   &=
   \<V \bcurl_h (u-v), \bcurl_h w\>_\G
   + \<T(u-v), [w]\>_{\gamma_h}
   \nonumber\\&\quad
   - \<[u-v], Tw\>_{\gamma_h}
   + \nu \<[u-v], [w]\>_{\gamma_h}
   \nonumber\\
   &=
   \<V \bcurlG (u-v), \bcurl_h w\>_\G
   + \<T(u-v), [w]\>_{\gamma_h}.
\end{align}
The first term on the right-hand side is bounded due to the mapping properties
\eqref{V_cont} of $V$ and \eqref{curl_cont_G} of $\bcurlG$:
\begin{align} \label{pf_main_2}
   \<V \bcurlG (u-v), \bcurl_h w\>_\G
   &\lesssim
   \|V \bcurlG (u-v)\|_{H^{1/2}(\G)} \|\bcurl_h w\|_{H^{-1/2}(\G)}
   \nonumber\\
   &\lesssim
   \|u-v\|_{H^{1/2}(\G)} \|\bcurl_h w\|_{H^{-1/2}(\G)}.
\end{align}
To bound $T(u-v)$ we employ a trace argument.
Let us consider an element $Q\in\CQ_h$  with edge
$e$, and a function $\psi\in H^s(Q)$ with $s>1/2$. We map $Q$
onto a fixed reference element $\hat Q$ and denote the mapped edge by
$\hat e$. Accordingly we denote the transformed function $\psi$ by $\hat\psi$.
Then, by transformation, using the trace theorem \eqref{trace} and the scaling
properties \eqref{scale3} with $s=0$ and \eqref{scale2}, we find that
\[
   \|\psi\|_{L^2(e)}^2
   \simeq
   h \|\hat\psi\|_{L^2(\hat e)}^2
   \lesssim
   \frac h{s-1/2} \|\hat\psi\|_{H^s(\hat Q)}^2
   \simeq
   \frac h{s-1/2}
   \Bigl( h^{-2}\|\psi\|_{L^2(Q)}^2 + h^{2s-2}|\psi|_{H^s(Q)}^2\Bigr).
\]
Applying this bound to the vector case, and recalling the
definition \eqref{T} of $T$, yields (with Sobolev-Slobodeckij semi-norm)
\[
   \|T(u-v)\|_{L^2(e)}^2
   \lesssim
   \frac 1{s-1/2}
   \Bigl( h^{-1}\|V\bcurlG (u-v)\|_{L^2(Q)}^2
        + h^{2s-1}|V\bcurlG (u-v)|_{H^s(Q)}^2\Bigr)
\]
and, summing over $e\in\CE_h$,
\be \label{pf_main_3}
   \|T(u-v)\|_{L^2(\gamma_h)}^2
   \lesssim
   \frac 1{s-1/2}
   \Bigl( h^{-1}\|V\bcurlG (u-v)\|_{L^2(\G)}^2
        + h^{2s-1}|V\bcurlG (u-v)|_{H^s(\CQ_h)}^2\Bigr).
\ee
By the mapping properties \eqref{V_cont} of $V$, the norm estimate \eqref{DD2}
and the boundedness \eqref{curl_cont_face_tilde} of $\bcurlS{\G_j}$ ($j=1,\ldots,L$)
we have that
\begin{align} \label{pf_main_4}
   \|V\bcurlG (u-v)\|_{L^2(\G)}^2
   &\lesssim
   \|\bcurlG (u-v)\|_{H^{-1}(\G)}^2
   \lesssim
   \sum_{j=1}^L \|\bcurlS{\G_j} (u-v)\|_{\tilde H^{-1}(\G_j)}^2
   \nonumber
   \\
   &\lesssim
   \sum_{j=1}^L \|u-v\|_{L^2(\G_j)}^2
   =
   \|u-v\|_{L^2(\G)}^2.
\end{align}
Similarly, first employing \eqref{DD3}, then the mapping properties \eqref{V_cont}
of $V$, the decomposition of norms \eqref{DD2}, the norm equivalence \eqref{equiv},
the boundedness of the surface curl \eqref{curl_cont_face_nontilde}, and
\eqref{DD1}, we obtain
\begin{align} \label{pf_main_5}
   |V\bcurlG (u-v)|_{H^s(\CQ_h)}^2
   &\lesssim
   |V\bcurlG (u-v)|_{H^s(\G)}^2
   \lesssim
   \|\bcurlG (u-v)\|_{H^{s-1}(\G)}^2
   \nonumber
   \\
   &\lesssim
   \sum_{j=1}^L \|\bcurlS{\G_j}(u-v)\|_{\tilde H^{s-1}(\G_j)}^2
   \lesssim
   \frac 1{(s-1/2)^2}
   \sum_{j=1}^L \|\bcurlS{\G_j}(u-v)\|_{H^{s-1}(\G_j)}^2
   \nonumber
   \\
   &\lesssim
   \frac 1{(s-1/2)^2}
   \sum_{j=1}^L \|u-v\|_{H^{s}(\G_j)}^2
   \lesssim
   \frac 1{(s-1/2)^2} \|u-v\|_{H^{s}(\G)}^2.
\end{align}
Here, we used implicitly that the Sobolev-Slobodeckij norm (started with) is
uniformly equivalent to the interpolation norm for $s$ close to $1/2$ so
that the continuity \eqref{curl_cont_face_nontilde} of $\bcurlG$ with respect to the
interpolation norm is applicable. In the final term we can switch back to the
Sobolev-Slobodeckij norm.

Combining \eqref{pf_main_3}, \eqref{pf_main_4} and \eqref{pf_main_5},
we obtain
\be \label{pf_main_6}
   \frac 1{\sqrt{\nu}} \|T(u-v)\|_{L^2(\gamma_h)}
   \lesssim
   \frac 1{\sqrt{\nu(s-1/2)}}
   \Bigl(
         h^{-1/2} \|u-v\|_{L^2(\G)}
         +
         \frac {h^{s-1/2}}{s-1/2} \|u-v\|_{H^s(\G)}
   \Bigr).
\ee
Combination of \eqref{pf_main_0}, \eqref{pf_main_1}, \eqref{pf_main_2} and
\eqref{pf_main_6} proves
\[
   \|u_{hp}-v\|_{H^{1/2}_\nu(\CQ_h)}
   \lesssim
   \|u-v\|_{H^{1/2}(\G)}
   +
   \frac 1{\sqrt{\nu(s-1/2)}}
   \Bigl(
      h^{-1/2} \|u-v\|_{L^2(\G)}
      +
      \frac {h^{s-1/2}}{s-1/2} \|u-v\|_{H^s(\G)}
   \Bigr)
\]
for any $v\in X_{hp}^0\cap C^0(\G)$.
By \eqref{pf_main_-2} and \eqref{pf_main_-1} this concludes the proof.

%%%%%%%%%%%%%%%%%%%%%%%%%%%%%%%%%%%%%%%%%%%%%%%%%%%%%%%%%%%%%%%%%%%%%%%%%%%%%%%%
\subsection{The problem on an open surface} \label{sec_open}

In the case of an open polyhedral surface the previous analysis carries over with\
few changes. The energy space (where one looks for the solution $u$ of \eqref{IE})
then is $\tilde H^{1/2}(\G)$. Moreover, the hypersingular and weakly singular operators
$W$ and $V$ are elliptic on $\tilde H^{1/2}(\G)$ and $\tilde H^{-1/2}(\G)$, respectively,
and the continuity \eqref{V_cont} becomes
\be \label{V_cont_open}
   V:\; \tilde H^{s-1}(\G) \to H^{s}(\G), \quad 0 \le s \le 1.
\ee
Since in this case the hypersingular operator is invertible one does not ask for $u$
to have integral mean value zero. Instead, its trace on the boundary $\partial\G$ of
$\G$ vanishes. Accordingly, functions of the discrete space $X_{hp}$ must comply
with this homogeneous boundary condition. In the corresponding discrete bilinear form
$a_h(\cdot,\cdot)$, jumps on edges lying on $\partial\G$ become traces.

To be precise, let us recall the notation $\CE_h$ for the set of edges of $\CQ_h$.
We distinguish between edges $\CE_h^b$ which are subsets of the boundary $\partial\G$,
and the rest $\CE_h^i$,
\[
   \CE_h = \CE_h^b \cup \CE_h^i.
\]
For $e\in\CE_h^i$ with $e=Q_1\cap Q_2$, $Q_1,Q_2\in\CQ_h$,
\[
   [v]|_e := (v|_{Q_1}-v|_{Q_2})|_e,
   %\quad \{v\}|_e := \frac 12(v|_{Q_1}+v|_{Q_2})|_e
\]
as before, and for $e\in \CE_h^b$
\[
   %\{v\}|_e := v|_e,\quad
   [v]|_e := v|_e.
\]
Accordingly we define unit tangential vectors on edges,
\be \label{t_b}
   \bt_e:= \left\{\begin{array}{ll}
      \bt_{Q_1}|_e & \quad\mbox{if}\ e\in\CE_h^i,\\
      \bt|_e       & \quad\mbox{if}\ e\in\CE_h^b
   \end{array}\right.
\ee
with $\bt$ being the unit tangential vector along $\partial\G$. The discontinuous
Galerkin boundary element method then reads:
{\it Find $u_{hp} \in X_{hp}$ such that}
\[
   a_h(u_{hp},v) = \<f, v\>_\G \quad\forall v\in X_{hp},
\]
and, instead of Theorem~\ref{thm_cea}, there holds the error estimate
\begin{multline} \label{cea_open}
   \|u-u_{hp}\|_{H^{1/2}_\nu(\CQ_h)}
   \le C
   \\
   \inf_{v\in X_{hp}\cap H^1_0(\G)}
   \Bigl\{
      \|u-v\|_{\tilde H^{1/2}(\G)}
      +
      \frac 1{\nu^{1/2}(s-1/2)^{3/2}}
      \Bigl(
         h^{-1/2} \|u-v\|_{L^2(\G)}
         +
         h^{s-1/2} \|u-v\|_{\tilde H^s(\G)}
      \Bigr)
   \Bigr\}
\end{multline}
with $s\in(1/2,\min\{r,1-\delta\}]$ for $r<1$ (since $u\not\in H^1(\G)$ in this
case) and norm
\[
   \|v\|_{H^s_\nu(\CQ_h)}^2 :=
   |v|_{H^s(\CQ_h)}^2  + \nu \|[v]\|_{L^2(\gamma_h)}^2.
\]
Let us note the few changes which are necessary in the proofs.
\begin{enumerate}
\item Proposition~\ref{prop_con}.
      In the case of an open surface we have $u\in \tilde H^s(\G)$ for any $s<1$
      so that the jumps of $u$ on $\gamma_h$ including its trace on $\partial\G$ vanish,
      and $V\bcurlG u$ is also continuous. The rest of proving consistency is identical.
\item Proposition~\ref{prop_ell}.
      In \eqref{pf_ell_1} one replaces the $H^{-1/2}(\G)$ with the $\tilde H^{-1/2}(\G)$-norm
      and notes that the injection of $\tilde H^{-1/2}(\G)$ in $H^{-1/2}(\G)$ is continuous.
\item Proof of \eqref{cea_open}. The steps from Section~\ref{sec_proof} carry over by noting
      that $[u]=u=0$ on $\partial\G$ and $[v]=v=0$ on $\partial\G$ for any $v\in X_{hp}$,
      and by replacing some ``non-tilde'' norms with their ``tilde'' counterparts.
\end{enumerate}

\begin{theorem} \label{thm_cea_plane}
In the case of a single face the error estimate can be improved to
\begin{multline*}
   \|u-u_{hp}\|_{H^{1/2}_\nu(\CQ_h)}
   \le C
   \\
   \inf_{v\in X_{hp}\cap H^1_0(\G)}
   \Bigl\{
      \|u-v\|_{\tilde H^{1/2}(\G)}
      +
      \frac 1{\sqrt{\nu(s-1/2)}}
      \Bigl(
         h^{-1/2} \|u-v\|_{L^2(\G)}
         +
         h^{s-1/2} \|u-v\|_{\tilde H^s(\G)}
      \Bigr)
   \Bigr\}.
\end{multline*}
\end{theorem}

\begin{proof}
This estimate follows as discussed above and by replacing \eqref{pf_main_5}
with the estimate
\[
   |V\bcurlG (u-v)|_{H^s(\CQ_h)}^2
   \lesssim
   |V\bcurlG (u-v)|_{H^s(\G)}^2
   \lesssim
   \|\bcurlG (u-v)\|_{\tilde H^{s-1}(\G)}^2
   \lesssim
   \|u-v\|_{\tilde H^{s}(\G)}^2
\]
where we used the continuity \eqref{V_cont_open} of $V$ and the continuity
\eqref{curl_cont_face_tilde} of $\bcurlG$ on the single face $\G$.
Then, instead of \eqref{pf_main_6}, we obtain
\[
   \frac 1{\sqrt{\nu}} \|T(u-v)\|_{L^2(\gamma_h)}
   \lesssim
   \frac 1{\sqrt{\nu(s-1/2)}}
   \Bigl(
         h^{-1/2} \|u-v\|_{L^2(\G)}
         +
         h^{s-1/2} \|u-v\|_{\tilde H^s(\G)}
   \Bigr).
\]
and the result follows.
\end{proof}

\begin{remark}
The improved estimate from Theorem~\ref{thm_cea_plane} leads to the
same convergence order of the $h$-version as stated by Corollary~\ref{cor_h}
(though $r<1$ by the reduced regularity).
For the $p$-version, however, it improves the estimate by a factor
of $\log^{-1}(1+p)$, cf.~Corollary~\ref{cor_p} and Remark~\ref{rem_p}.
\end{remark}

%%%%%%%%%%%%%%%%%%%%%%%%%%%%%%%%%%%%%%%%%%%%%%%%%%%%%%%%%%%%%%%%%%%%%%%%%%%%%%%%
\section{Numerical results} \label{sec_num}

We consider the model problem (\ref{IE}) with $f=1$ on $\G=(0,1) \times (0,1)$
and use uniform meshes $\CQ_h$ on $\G$ consisting of squares.

Since the exact solution $u$ of \eqref{IE} is unknown, the error
\[
   |u-u_{hp}|_{H^{1/2}(\CQ_h)}^2
   +
   \nu
   \|[u_{hp}]\|_{L^2(\gamma_h)}^2
\]
cannot be computed directly, except for $\|[u_{hp}]\|_{L^2(\gamma_h)}$ which is straightforward
to implement.
As in previous papers on non-conforming approximations of hypersingular operators
(see, e.g., \cite{ChoulyH_NDD}) we approximate an upper bound to the semi-norm 
$|u-u_{hp}|_{H^{1/2}(\CQ_h)}$. First, note that there holds
\[
   |u-u_{hp}|_{H^{1/2}(\CQ_h)}^2 \lesssim
   \<V \bcurl_h (u - u_{hp}), \bcurl_h (u - u_{hp}) \>_\G,
\]
cf. \eqref{pf_ell_3}.
We find that
\begin{align*}
   |u-u_{hp}|_{H^{1/2}(\CQ_h)}^2
   &\lesssim
   \<V \bcurlG u, \bcurlG u \>_\G + \<V \bcurl_h u_{hp}, \bcurl_h u_{hp} \>_\G
   - 2 \<V \bcurlG u, \bcurl_h u_{hp} \>_\G.
\end{align*}
Let us analyze the three terms on the right-hand side.\\
(i) Since $u\in\tilde H^{1/2}(\G)$, there holds
\[
   \<V \bcurlG u, \bcurlG u \>_\G = \<W u, u\>_\G.
\]
The energy norm (squared) $\<Wu,u\>_\G$ of $u$ can be approximated through
extrapolation, in the following denoted by $\|u\|_{\mathrm ex}^2$,
see \cite{ErvinHS_93_hpB}.\\
(ii) By construction of $u_{hp}$ \eqref{DGBEM},
\[
   \<V \bcurl_h u_{hp}, \bcurl_h u_{hp} \>_\G = \<f,u_{hp}\>_\G - \nu \|[u_{hp}]\|_{L^2(\gamma_h)}^2.
\]
(iii) By consistency (see Proposition~\ref{prop_con}),
\[
   \<V \bcurlG u, \bcurl_h u_{hp} \>_\G = \<f,u_{hp}\>_\G - \<Tu,[u_{hp}]\>_{\gamma_h}.
\]
For our example, $u\in \tilde H^{1-\epsilon}(\G)$ for any $\epsilon>0$, so that
$Tu\in H^{1/2-\epsilon}(\gamma_h)$ for any $\epsilon>0$. By duality,
\[
   \<Tu,[u_{hp}]\>_{\gamma_h}
   \le
   \|Tu\|_{H^{1/2-\epsilon}(\gamma_h)} \|[u_{hp}]\|_{H^{-1/2+\epsilon}(\gamma_h)}.
\]
Combining (i)--(iii) we conclude that for any $\epsilon>0$ there holds (with a constant depending
on $\epsilon$)
\[
   |u-u_{hp}|_{H^{1/2}(\CQ_h)}^2
   \lesssim
   |\<W u, u\>_\G - \<f,u_{hp}\>_\G |
   + \nu \|[u_{hp}]\|_{L^2(\gamma_h)}^2 + \|[u_{hp}]\|_{H^{-1/2+\epsilon}(\gamma_h)}.
\]
Since $u\in H^{1/2-\epsilon}(\gamma_h)$ (and considering the singularities of $u$)
we expect that
\[
   \|[u_{hp}]\|_{H^{-1/2+\epsilon}(\gamma_h)} \lesssim h^{1-2\epsilon} p^{-2(1-2\epsilon)}
   \qquad\mbox{and}\qquad
   \|[u_{hp}]\|_{L^2(\gamma_h)}^2 \lesssim h^{1-2\epsilon} p^{-2(1-2\epsilon)}
\]
are of the same order (note the different exponents of the norms),
so that for our numerical experiments we take
\begin{equation*}
   \Bigl(\bigl|\|u\|_{\rm ex}^2 - \<f,u_{hp}\>_\G\bigr|^{1/2}
         + \, \|[u_{hp}]\|_{L^2(\gamma_h)}
   \Bigr)/\|u\|_{\rm ex}
\end{equation*}
as a computable and reasonable measure for an upper bound of the error
$\|u-u_{hp}\|_{H^{1/2}(\CQ_h)}$ normalized by
$\|u\|_{\tilde H^{1/2}(\G)}\approx \|u\|_{\rm ex}$.

Below we present numerical results for the two contributions
\[
   \bigl|\|u\|_{\rm ex}^2 - \<f,u_{hp}\>_\G\bigr|^{1/2}/\|u\|_{\rm ex}
\]
(referred to as ``$H^{1/2}$'' error in the figures) and
\[
   \|[u_{hp}]\|^{1/2}_{L^2(\gamma)}/\|u\|_{\rm ex}
\]
(referred to as ``$L^2$'' error) which measures the jumps, a measure for
non-conformity.

We first study the influence of $\nu$. In Figure~\ref{fig_sol} four different approximations
are presented, all on the mesh of $25$ elements and polynomial degree $p=3$.
The first three pictures (a)--(c) show $u_{hp}$ for increasing $\nu$ ($\nu=0.1, 1, 10$) and
the last picture (d) shows the conforming approximation. One sees that with increasing
$\nu$ discontinuities quickly disappear, and $u_{hp}$ with $\nu=10$ is visually not distinguishable
from the conforming counterpart. Note, in particular, how the trace of $u_{hp}$ on the
boundary of $\G$ approaches $0$ when $\nu$ increases. Nevertheless, even for $\nu=100$,
discontinuities are still present and converge at the predicted rate,
see Figure~\ref{fig_hp2_nu100} below.

Now let us study convergence rates. In our case of an open surface,
the strongest singularities in \eqref{reg} are edge singularities with $\mu=1/2$
so that the $h$-version for any polynomial degree should converge like $h^{1/2}$ and
the $p$-version like $p^{-1}$, see~\cite{BespalovH_05_pBE,BespalovH_08_hpB}, and
\cite{HeuerMS_99_ECB} for numerical results.

In Figures~\ref{fig_hp1_nu1}, \ref{fig_hp1_nu20} and \ref{fig_hp1_nu100} we present
on double-logarithmic scales the relative ``$H^{1/2}$'' errors with $\nu=1,20,100$,
respectively, for the polynomial degrees $p=1,3,5$ in all cases. We also show the
errors for the conforming method ($p=1$) and, for comparison, the curve $h^{1/2}$.
The latter curve indicates the expected rate of convergence (cf. Corollaries~\ref{cor_h} and
\ref{cor_hp}) and, indeed, the numerical results show at least this convergence rate.
Note, however, that in Figure~\ref{fig_hp1_nu1} the curve for $p=1$ is below the ones for
$p=3,5$, whereas for a conforming method increasing $p$ reduces the error. In this case
$\nu=1$ is small and we explain this phenomenon by the fact that we are dealing with
a saddle point formulation where on cannot expect convergence like that of a projection method
(as it is the case with conforming approximations of this operator).
Indeed, when $\nu$ increases, the expected order of the curves is re-established.
First, for $\nu=20$ in Figure~\ref{fig_hp1_nu20}, we have a pre-asymptotic super-convergence
and then, for $\nu=100$ in Figure~\ref{fig_hp1_nu100}, the curves for the different degrees
have the right order (highest degree gives smallest error).

In Figures~\ref{fig_hp2_nu1}, \ref{fig_hp2_nu20} and \ref{fig_hp2_nu100} the corresponding
``$L^2$'' errors of the jumps are given. They are all of the right convergence order
$h^{1/2}$ and in the right order.

In Figures~\ref{fig_p1} and \ref{fig_p2} we plot the errors (``$H^{1/2}$'' resp. ``$L^2$'')
obtained by the $p$-version and for different values of $\nu$ ($\nu=1,10,50,100$) along
with the curve $p^{-1}$ which represents the expected convergence rate, cf.~Corollary~\ref{cor_p}.
For the given range of degrees of freedom, smaller $\nu$ does not result in the
expected convergence order of the ``$H^{1/2}$'' errors
(at least not of our upper bound) and we conclude that we are
in a pre-asymptotic range. For higher $\nu$ ($\nu=50,100$) the optimal order is observed.
In contrast, the ``$L^2$'' convergence of the jumps is the optimal one for all $\nu$
and, not surprisingly, larger $\nu$ leads to smaller errors.

\begin{figure}[htb]
  \begin{center}
      \subfigure[$\nu=0.1$]{\includegraphics[width=0.45\textwidth]{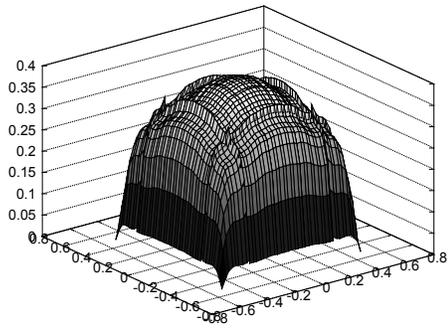}}
      \subfigure[$\nu=1$]{\includegraphics[width=0.45\textwidth]{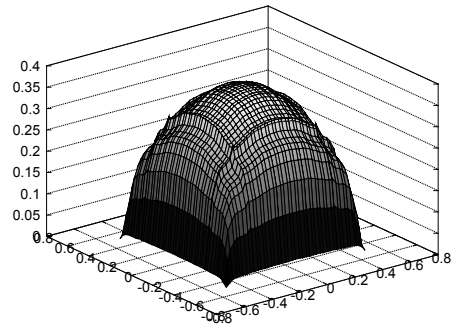}}

      \subfigure[$\nu=10$]{\includegraphics[width=0.45\textwidth]{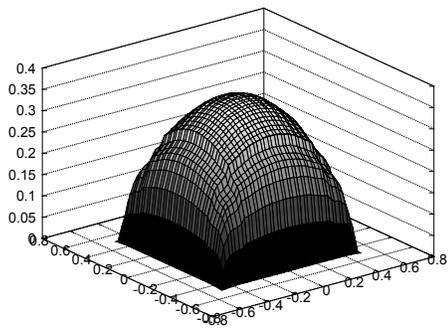}}
      \subfigure[conforming]{\includegraphics[width=0.45\textwidth]{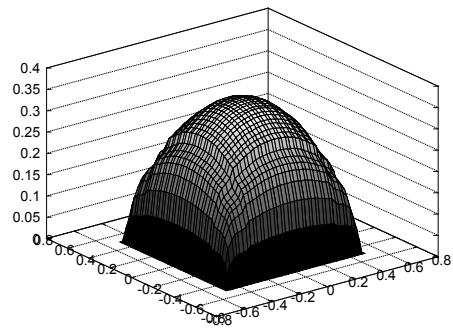}}
    \caption{Approximation with $5\times 5$ elements and $p=3$.}
    \label{fig_sol}
  \end{center}
\end{figure}

\begin{figure}[htb]
\centering
\includegraphics[width=0.75\textwidth]{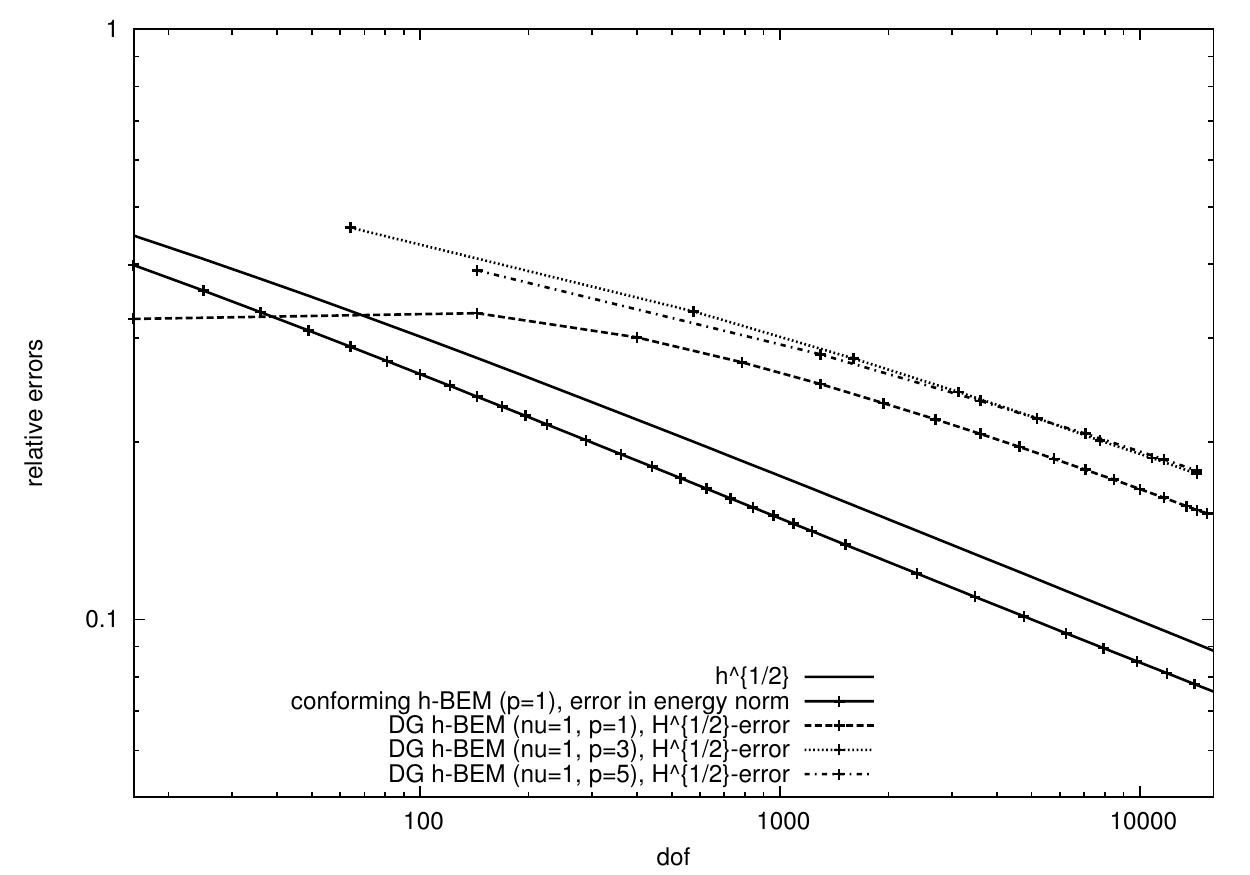} 
\caption{$h$-version with degrees $p=1,3,5$ and $\nu=1$,
         relative ``$H^{1/2}$'' errors.
         Comparison with conforming BEM.}
\label{fig_hp1_nu1}
\end{figure}

\begin{figure}[htb]
\centering
\includegraphics[width=0.75\textwidth]{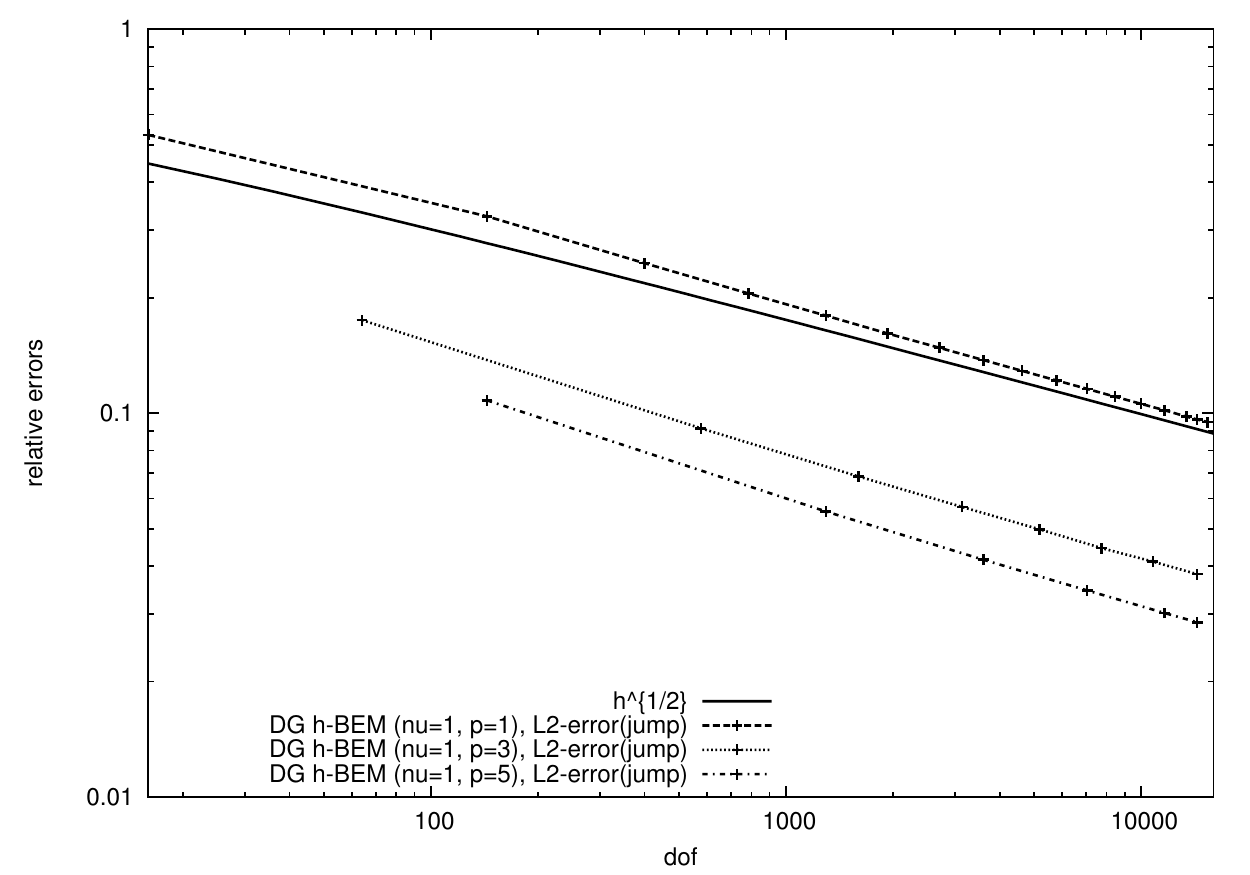} 
\caption{$h$-version with degrees $p=1,3,5$ and $\nu=1$,
         relative ``$L^2$'' errors.}
\label{fig_hp2_nu1}
\end{figure}

\begin{figure}[htb]
\centering
\includegraphics[width=0.75\textwidth]{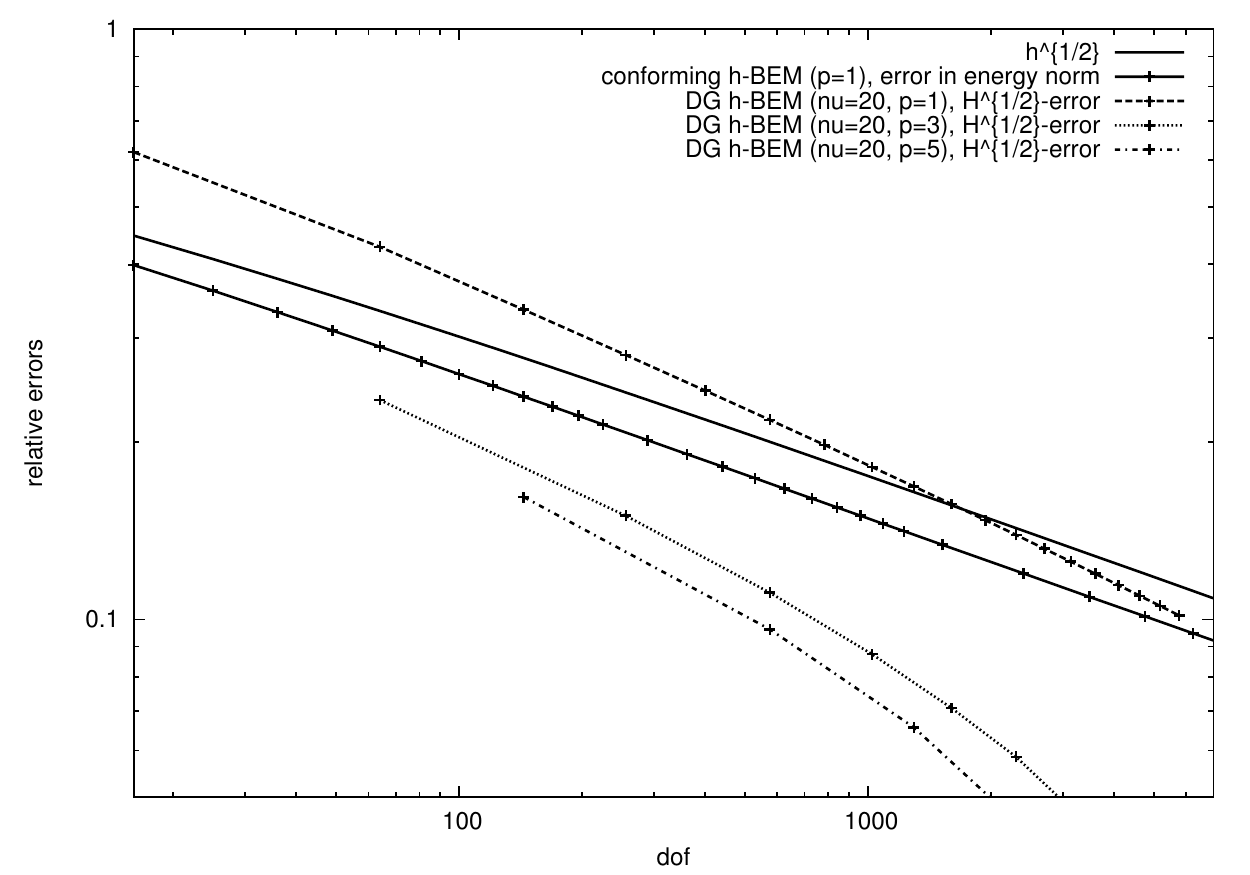} 
\caption{$h$-version with degrees $p=1,3,5$ and $\nu=20$,
         relative ``$H^{1/2}$'' errors.
         Comparison with conforming BEM.}
\label{fig_hp1_nu20}
\end{figure}

\begin{figure}[htb]
\centering
\includegraphics[width=0.75\textwidth]{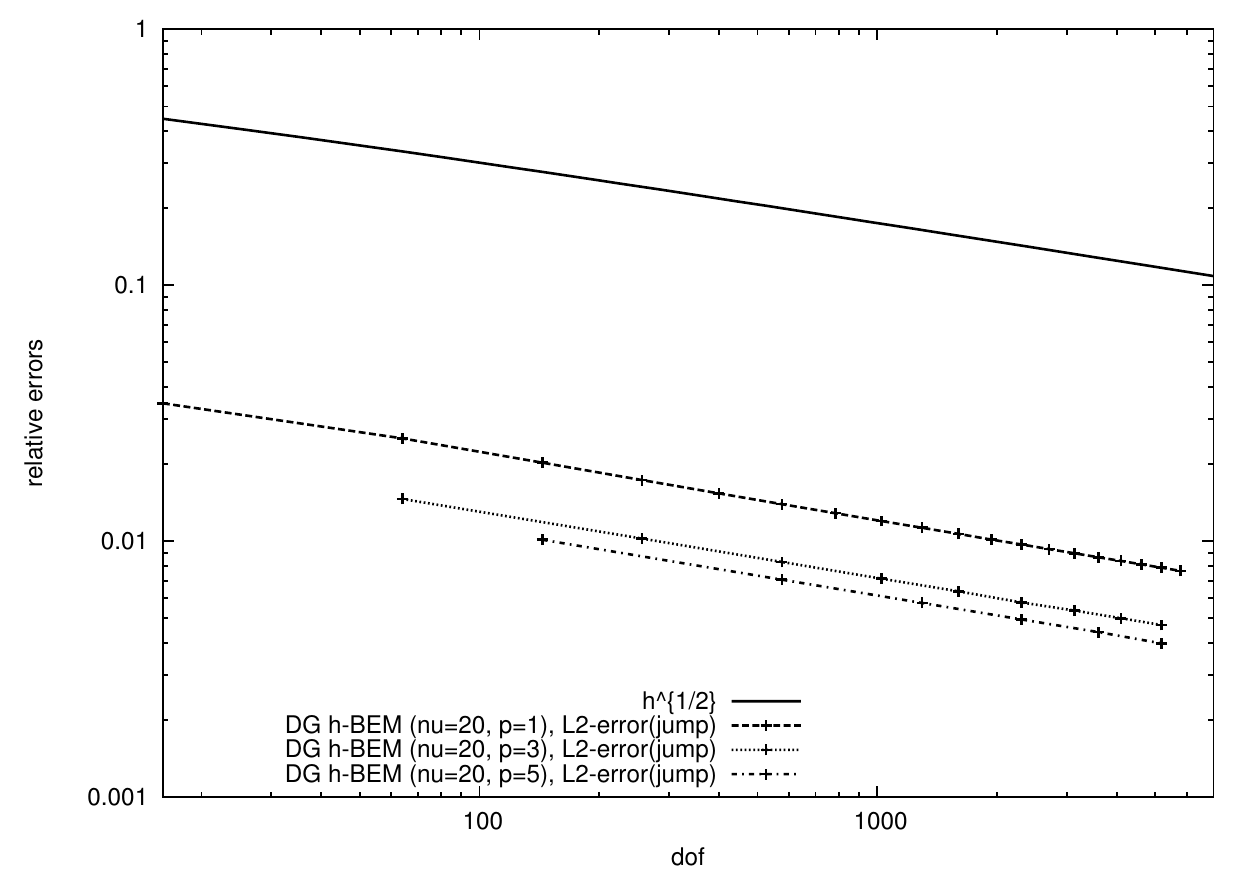} 
\caption{$h$-version with degrees $p=1,3,5$ and $\nu=20$,
         relative ``$L^2$'' errors.}
\label{fig_hp2_nu20}
\end{figure}

\begin{figure}[htb]
\centering
\includegraphics[width=0.75\textwidth]{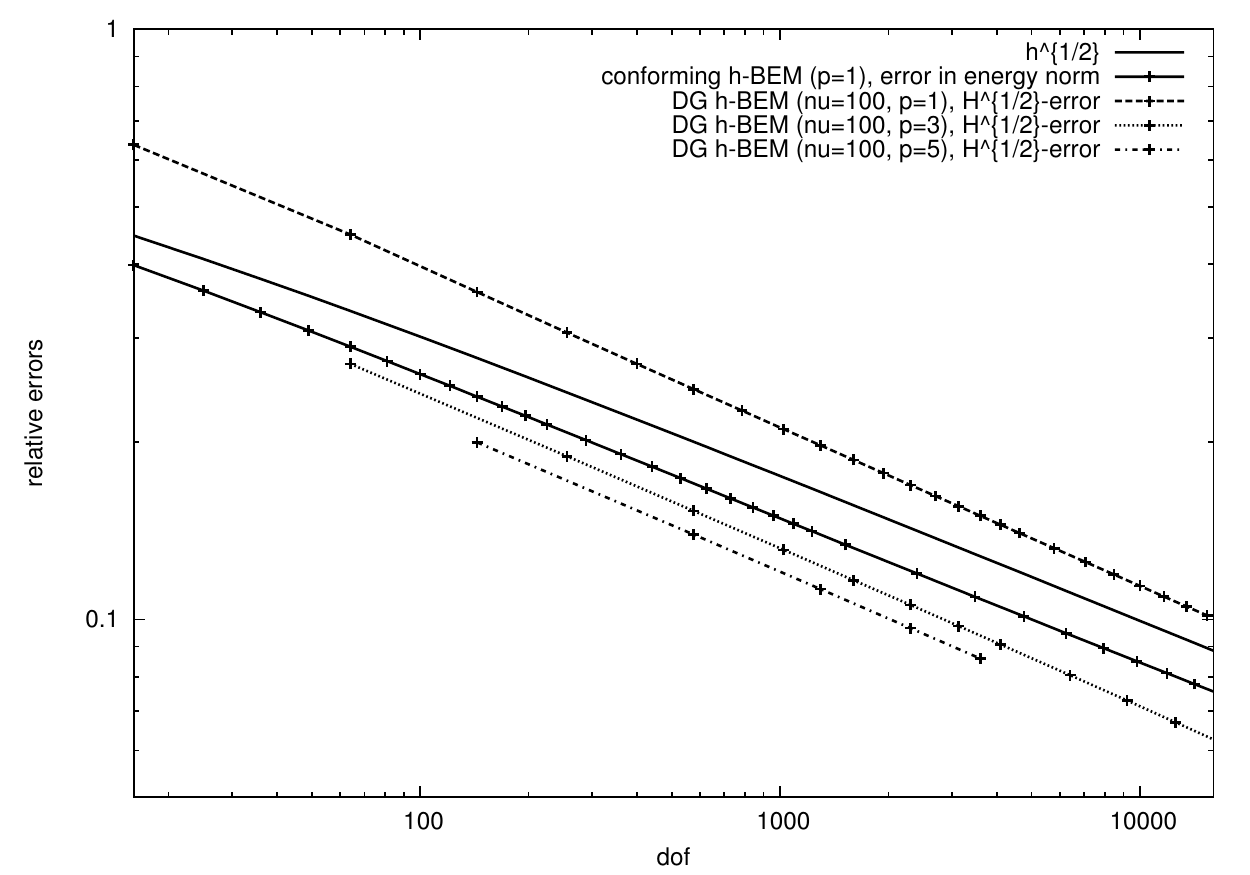} 
\caption{$h$-version with degrees $p=1,3,5$ and $\nu=100$,
         relative ``$H^{1/2}$'' errors.
         Comparison with conforming BEM.}
\label{fig_hp1_nu100}
\end{figure}

\begin{figure}[htb]
\centering
\includegraphics[width=0.75\textwidth]{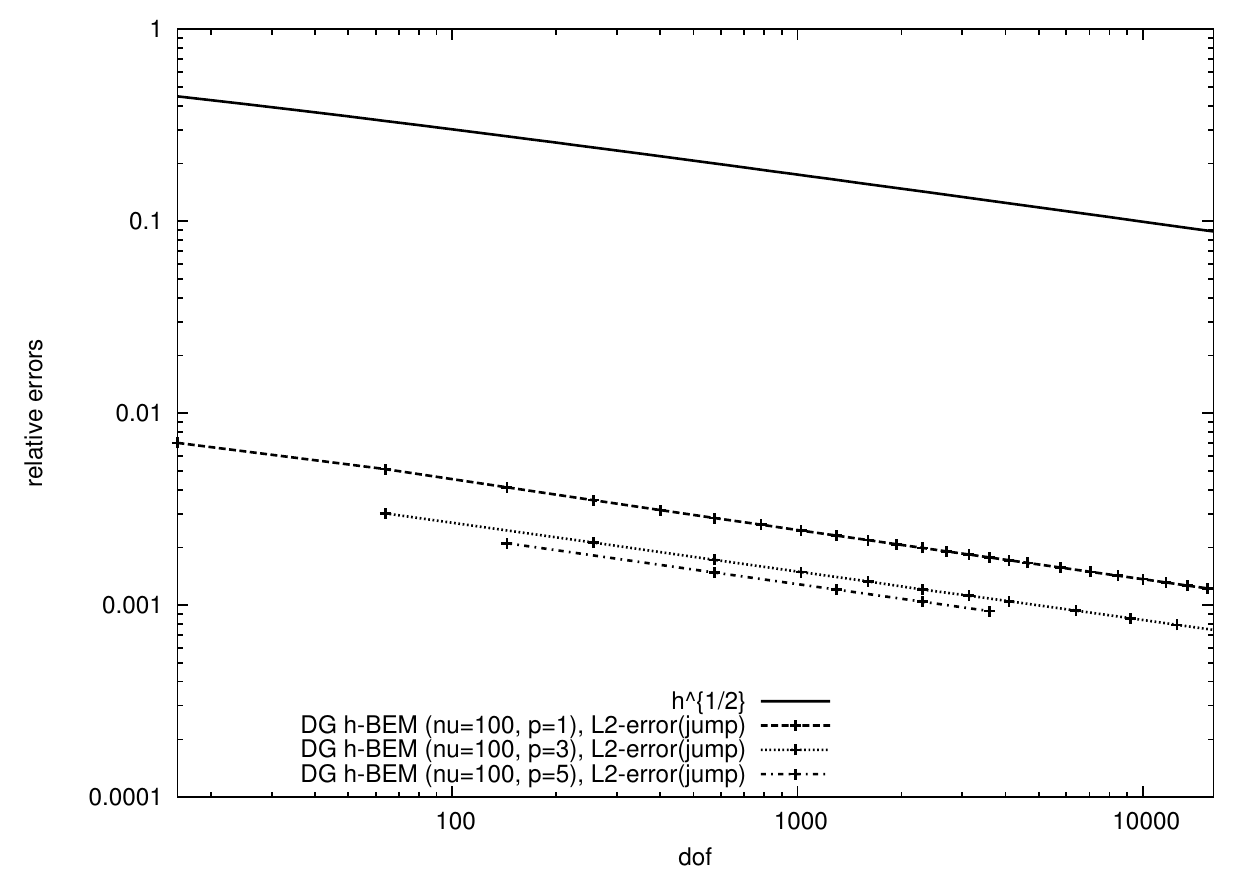} 
\caption{$h$-version with degrees $p=1,3,5$ and $\nu=100$,
         relative ``$L^2$'' errors.}
\label{fig_hp2_nu100}
\end{figure}

\begin{figure}[htb]
\centering
\includegraphics[width=0.75\textwidth]{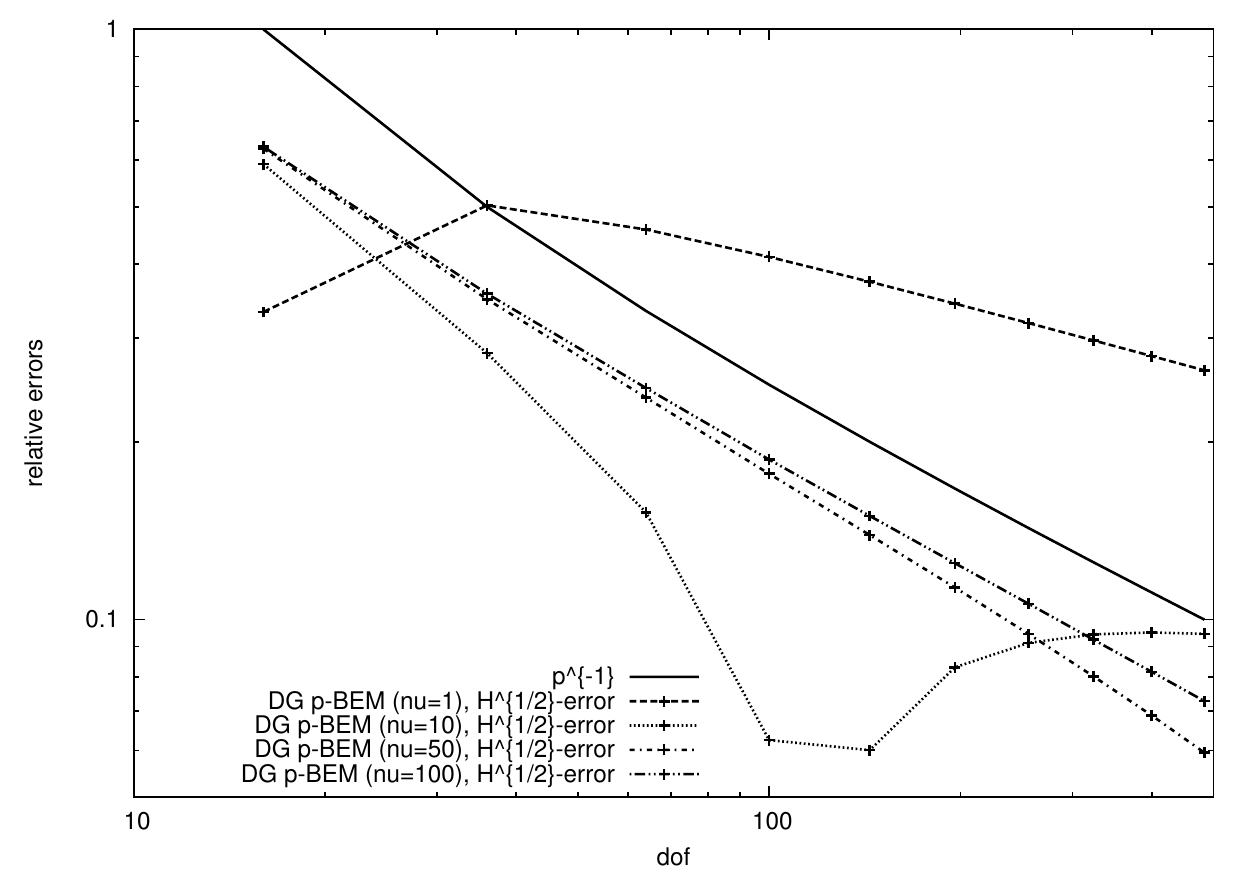} 
\caption{$p$-version with $4$ elements and $\nu=1,10,50,100$,
         relative ``$H^{1/2}$'' errors.}
\label{fig_p1}
\end{figure}

\begin{figure}[htb]
\centering
\includegraphics[width=0.75\textwidth]{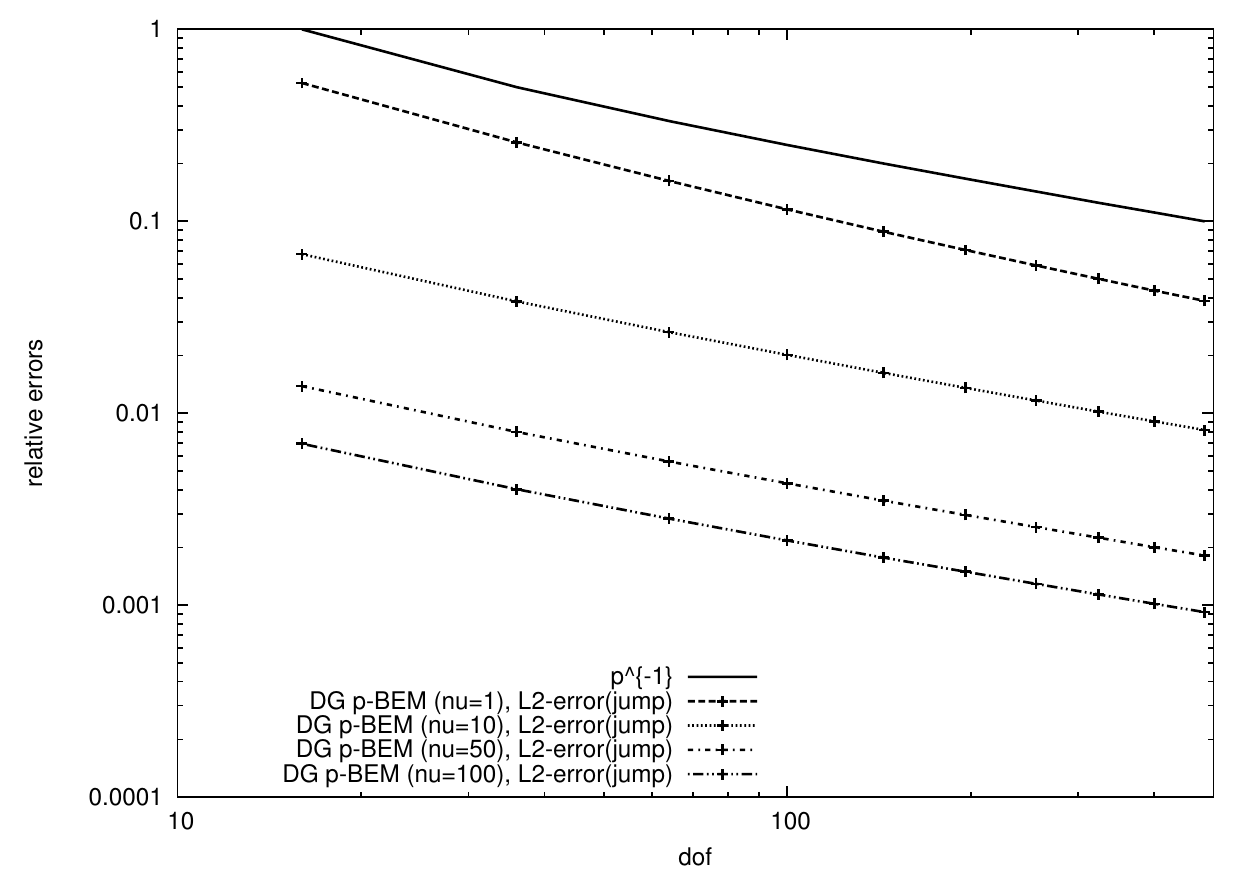} 
\caption{$p$-version with $4$ elements and $\nu=1,10,50,100$,
         relative ``$L^2$'' errors.}
\label{fig_p2}
\end{figure}

%%%%%%%%%%%%%%%%%%%%%%%%%%%%%%%%%%%%%%%%%%%%%%%%%%%%%%%%%%%%%%%%%%%%%%%%%%%%%%%%
\clearpage
\bibliographystyle{siam}
%\bibliography{biblio}
\bibliography{../../bib/bib,../../bib/heuer,../../bib/fem}

\end{document}